\documentclass[reqno]{amsart}
\usepackage{amsmath, amsthm, amssymb, xypic}
\usepackage{hyperref}
\usepackage[nobysame,non-sorted-cites]{amsrefs} 
\usepackage{enumitem}
\setenumerate{label=\textnormal{(\arabic*)}} 
\usepackage{graphicx}
\usepackage{tikz}
	\usetikzlibrary{cd}

\usepackage{color}

\xyoption{rotate}

\newcommand{\cat}[1]{\ensuremath{\mathsf{#1}}}

\newcommand{\op}{\ensuremath{^\mathrm{op}}}

\newcommand{\ctens}{\mathbin{\widehat{\otimes}}}

\DeclareMathOperator{\Top}{\cat{Top}}

\DeclareMathOperator{\Alg}{\cat{Alg}}

\DeclareMathOperator{\PCAlg}{\cat{PCAlg}}

\DeclareMathOperator{\Coalg}{\cat{Coalg}}

\DeclareMathOperator{\Spec}{Spec}
\renewcommand{\O}{\mathcal{O}}

\newcommand{\Vect}{\cat{Vect}}
\DeclareMathOperator{\Hom}{Hom}
\DeclareMathOperator{\End}{End}
\DeclareMathOperator{\id}{id}

\DeclareMathOperator{\Dist}{Dist}

\newcommand{\F}{\mathcal{F}}
\newcommand{\eps}{\varepsilon}

\newcommand{\setS}{\mathcal{S}}

\newcommand{\invlim}{\varprojlim}
\newcommand{\dirlim}{\varinjlim}

\newcommand{\Z}{\mathbb{Z}}

\newcommand{\A}{\mathbb{A}}

\DeclareMathOperator{\lPC}{\!-\cat{PC}}

\DeclareMathOperator{\PC}{\cat{PC}}
\DeclareMathOperator{\CF}{\cat{CF}}
\DeclareMathOperator{\lComod}{\!-\cat{Comod}}

\newcommand{\del}{\partial}

\DeclareMathOperator{\ch}{char}

\newcommand{\separate}{\bigskip}

\numberwithin{equation}{section}

\theoremstyle{plain}
\newtheorem{theorem}[equation]{Theorem}
\newtheorem{corollary}[equation]{Corollary}
\newtheorem{lemma}[equation]{Lemma}
\newtheorem{proposition}[equation]{Proposition}

\theoremstyle{definition}

\newtheorem{example}[equation]{Example}
 
\newtheorem{remark}[equation]{Remark}

\begin{document}
\title{Dual coalgebras of twisted tensor products}
\author{Manuel L. Reyes}
\address{Department of Mathematics\\ University of California, Irvine\\
  340 Rowland Hall\\ Irvine, CA 92697--3875\\ USA}
\email{mreyes57@uci.edu}
\urladdr{https://www.math.uci.edu/~mreyes/}
\date{January 16, 2025}
\thanks{This work was supported by NSF grant DMS-2201273.}
\keywords{finite dual coalgebra, twisted tensor product, cotwisted tensor product, topological tensor product, continuous dual}
\subjclass[2020]{
Primary: 16S35, 
16T15, 
18M05, 
Secondary: 
16S38, 
16S40, 
16S80
}

\begin{abstract}
We investigate cases where the finite dual coalgebra of a twisted tensor product of two algebras is a cotwisted tensor product of their respective finite dual coalgebras. This is achieved by interpreting the finite dual as a topological dual; in order to prove this, we show that the continuous dual is a strong monoidal functor on linearly topologized vector spaces whose open subspaces have finite codimension. We describe a sufficient condition for the result on finite dual coalgebras to be applied, and we specialize this condition to particular constructions including Ore extensions, smash product algebras, and bitwisted tensor products of bialgebras.
\end{abstract}

\maketitle

\section{Introduction}
\label{sec:intro}

Let $k$ be an arbitrary field, and let $A$ be a $k$-algebra.
The finite dual~\cite[1.3]{heynemansweedler} of $A$ is a well-known coalgebra associated to $A$ that is defined as follows. Let $\F(A)$ denote the family of all ideals of finite $k$-codimension in $A$. Then $A^\circ$ is the subspace of the dual space $A^*$ consisting of all functionals that vanish on some ideal $I \in \F(A)$. The coalgebra structure of $A$ can be described up to isomorphism as the direct limit
\[
A^\circ \cong \dirlim_{I \in \F(A)} (A/I)^*,
\]
and the assignment $A \mapsto A^\circ$ forms a functor $(-)^\circ \colon \Alg \to \Coalg$ from $k$-algebras to $k$-coalgebras.
In~\cite{reyes:qmax}, the finite dual coalgebra was discussed from a geometric perspective as a noncommutative functorial substitute for the maximal spectrum of an algebra. 
Viewing the finite dual as a spectral invariant of an algebra highlights the important problem of finding techniques to compute $A^\circ$ for a given algebra $A$.
This has been addressed for certain examples of Hopf algebras, including for instance~\cite{takeuchi:sl2, chinmusson:duality, jahn:thesis, couto:thesis, geliu:dual, liliu:duals, browncoutojahn:dual}. However, it seems that outside of~\cite{reyes:qmax} this problem has not been studied for general algebras in the absence of a bialgebra structure. 

The goal of this paper is to provide a new tool to assist with the computation of finite dual coalgebras of algebras that arise from familiar operations.
Many common constructions to build noncommutative algebras out of two subalgebras---such as skew polynomial rings, twisted group algebras, and smash product algebras---are instances of a \emph{twisted tensor product}~\cite{csv:twisted} construction. Such an algebra $A \otimes_\tau B$ is constructed from two algebras $A$ and $B$ and a linear \emph{twisting} map $\tau \colon B \otimes A \to A \otimes B$ that satisfies certain nice properties (see Subsection~\ref{sub:crossed product} for more details).
Recent papers such as~\cite{gkm:discriminant, connergoetz:property, sheplerwitherspoon:twisted, connergoetz:classification} have studied a number of interesting aspects of this construction, and it seems that it could be useful to have spectral techniques to aid in this analysis.
Thus we have focused our attention on finding a suitable method to compute the finite dual coalgebra of a twisted tensor product.

There is a formally dual \emph{cotwisted tensor product} construction $C \otimes^\phi D$ for coalgebras $C$ and $D$ with a suitable linear \emph{cotwisting} map $\phi \colon C \otimes D \to D \otimes C$, defined in~\cite{cimz:smash} with different terminology.
Thus one might naturally wish for the finite dual to interchange these operations in a formula such as
\begin{equation}\label{eq:goal}
(A \otimes_\tau B)^\circ \cong A^\circ \otimes^{\phi} B^\circ.
\end{equation}
However it is easy to find instances where this fails, as in Example~\ref{ex:not continuous} below. Nevertheless, in this paper we are able to find a sufficient condition on the twisting map for an isomorphism of the above form to hold.

Reasoning in terms of functoriality, the isomorphism~\eqref{eq:goal} would seem more reasonable if one could extend the domain of the functor $(-)^\circ$ to a category in which the twisting map $\tau$ is a morphism, which would allow us to  define $\phi = \tau^\circ$. We employ this strategy to obtain a result of the form~\eqref{eq:goal}, with the larger category $\Top_k$ being that of linearly topologized $k$-vector spaces, where $k$ is endowed with the discrete topology. This category is discussed in Subsection~\ref{sub:monoidal} below, along with a topological tensor product $- \otimes^! -$ (whose underlying vector space is the ordinary tensor product) that makes $\Top_k$ into a symmetric monoidal category. 
With a slight abuse of notation, we denote the continuous dual functor as
\[
(-)^\circ = \Top_k(-,k) \colon \Top_k\op \to \Vect_k.
\]
Our notation is justified as follows.

Let $\CF_k$ denote the full subcategory of $\Top_k$ consisting of those topological vector spaces whose open subspaces all have finite codimension. This is in fact a monoidal subcategory of $\Top_k$.
Every algebra $A$ can be given the structure of a topological algebra when equipped with the linear topology whose open ideals are exactly those in $\F(A)$; we call this the \emph{cofinite topology} of $A$. In this way, we may view $A$ a monoid object of $(\CF_k, \otimes^!, k)$. As we explain in Subsection~\ref{sub:finite dual}, the finite dual and the topological dual are equal as subspaces of the full linear dual:
\[
A^\circ = \Top_k(A,k) \subseteq A^*.
\]  
This allows us to state the first major result of the paper.

\begin{theorem}[Theorems~\ref{thm:tensor dual}, \ref{thm:continuous equals sweedler}] \label{thm:A}
Let $(-)^\circ = \Top_k(-,k)\op$ denote the continuous dual functor as above.
\begin{enumerate}
\item The continuous dual restricts to a strong monoidal functor $(-)^\circ \colon (\CF_k, \otimes^!, k) \to (\Vect_k, \otimes, k)$, so that if $E$ and $F$ have cofinite topologies then
\[
(E \otimes F)^\circ \cong E^* \otimes F^*.
\]
\item Let $A$ be an algebra with multiplication $m \colon A \otimes A \to A$. If $A$ is endowed with its cofinite topology, then  $m$ is continuous as a morphism $A \otimes^! A \to A$, and the comultiplication of the finite dual $A^\circ$ is the composite 
\[
\begin{tikzcd}
A^\circ \ar[r, "m^\circ"] & (A \otimes^! A)^\circ \ar[r, "\sim"] & A^\circ \otimes A^\circ.
\end{tikzcd}
\]
\end{enumerate}
\end{theorem}

This topological point of view allows us to formulate the second main result of the paper, which achieves our goal of finding an isomorphism of the form~\eqref{eq:goal}. After reviewing twisted tensor products and cotwisted tensor products, we are able to establish the following in Section~\ref{sec:twisted tensor}.

\begin{theorem}[Theorem~\ref{thm:twisted dual}] \label{thm:B}
Let $A$ and $B$ be algebras with a twisting map $\tau \colon B \otimes A \to A \otimes B$. Endow $A$ and $B$ with their cofinite topologies. If the twisting map is continuous as a map $B \otimes^! A \to A \otimes^! B$, then the linear map $A^\circ \otimes B^\circ \to B^\circ \otimes A^\circ$ induced by $\tau^\circ$ is a cotwisting map, and we have an isomorphism of coalgebras
\[
(A \otimes_\tau B)^\circ \cong A^\circ \otimes^{\tau^\circ} B^\circ.
\]
\end{theorem}

\separate

In Section~\ref{sec:applications} we describe a sufficient condition (Theorem~\ref{thm:centralize}) for the twisting map $\tau$ above to be continuous, so that Theorem~\ref{thm:B} applies. We then describe situations where this applies to specific constructions such as Ore extensions, smash product algebras, and bitwisted tensor products of bialgebras. 
Finally, in Section~\ref{sec:examples}, we describe several examples of twisted tensor products of the form $k[x] \otimes_\tau k[y]$ whose dual coalgebras decompose as cotwisted tensor products. This includes the quantum plane and quantized Weyl algebras at roots of unity, as well as the Jordan plane and Weyl algebra in positive characteristic.

\separate

\textbf{Acknowledgments.}
I thank Hongdi Huang for a question that led to Corollary~\ref{cor:crossed bialgebra}. Additionally, I am deeply grateful to an anonymous referee for a generous number of corrections, suggestions, and references to the literature that have greatly improved the readability of this paper.

\section{Continuous duality between algebras and coalgebras}\label{sec:duality}

This section describes some elements of duality theory relating algebras and coalgebras. This includes discussions of linearly topologized vector spaces, topologized tensor products, and continuous duality. 

Let $k$ be a field, which is arbitrary throughout this paper unless explicitly stated otherwise. Unadorned tensor symbols $- \otimes -$ denote tensor over $k$. In this paper, all algebras are unital and associative, all coalgebras are counital and coassociative, and morphisms of these objects preserve the (co)unit. 
We let $\Alg = \Alg_k$ denote the category of $k$-algebras and $\Coalg = \Coalg_k$ denote the category of $k$-coalgebras.

\subsection{Continuous duality and topological tensor products}\label{sub:monoidal}

We begin by recalling some elements of topological rings and their modules.
Let $R$ be a topological ring. A topological left $R$-module $M$ is \emph{linearly topologized} if it has a neighborhood basis of zero consisting of open submodules. We say $R$ is \emph{left linearly topologized} if ${}_R R$ is a linearly topologized module, and \emph{right linearly topologized rings} are similarly defined. We say that $R$ is \emph{linearly topologized} if it is both left and right linearly topologized, or equivalently~\cite[Remark~2.19]{imr:diagonal}, if it has a neighborhood basis of zero consisting of open ideals.

We consider our field $k$ as a topological field with the discrete topology. Let $R$ be a topological $k$-algebra.
Recall that a topological left $R$-module $M$ is said to be \emph{pseudocompact} if it satisfies the following equivalent conditions:
\begin{itemize}
\item $M$ is separated (i.e., Hausdorff), complete, and has a neighborhood basis of zero consisting of open submodules of finite codimension;
\item $M$ is an inverse limit (in the category of topological modules) of discrete,
finite-dimensional modules;
\item The natural homomorphism $M\to \varprojlim M/N$, where $N$ ranges over the open submodules 
of finite codimension in $M$, is an isomorphism of topological modules.
\end{itemize}
The topology on such a module $M$ is induced by viewing it as a subspace $M \cong \varprojlim M/N \subseteq \prod M /N$ of the product space $\prod M/N$ where again $N$ ranges over the open submodules of finite codimension, and each finite-dimensional $M/N$ is endowed with the discrete topology~\cite[Lemma~2.17]{imr:diagonal}.

Let $R \lPC$ denote the category of pseudocompact left $R$-modules with continuous homomorphisms. In particular, taking $R = k$, we have the category $\PC_k := k \lPC$ of \emph{pseudocompact vector spaces}. These are topological vector spaces that are inverse limits of finite-dimensional discrete vector spaces.

We let $\Top_k$ denote the category of linearly topologized $k$-vector spaces with continuous $k$-linear maps. Then $\PC_k$ forms a full subcategory of $\Top_k$. We will identify $\Vect_k$ with the full subcategory of discrete topological vector spaces within $\Top_k$. We say that a linearly topologized vector space $E$ is \emph{cofinite} if its open subspaces all have finite codimension in $E$. We let $\CF_k$ denote the full subcategory of $\Top_k$ consisting of cofinite spaces. Then we have the following inclusions of full subcategories:
\[
\PC_k \subseteq \CF_k \subseteq \Top_k.
\]

The completion of a separated topological vector space is a well-known construction~\cite[Section~7]{warner}, but we will require a straightforward extension to spaces that are not necessarily separated. For a linearly topologized space $E$, we let $\setS(E)$ denote the family of open subspaces of $E$, which forms a neighborhood basis of zero in $E$. The \emph{separation} of $E$,
\[
E_s := \left. E \middle/ \overline{\{0\}} \right. 
= \left. E \middle/ \left( \bigcap \setS(E) \right) \right.,
\]
is universal among all separated spaces with a continuous surjection from $E$. Then for such $E$, we can define the \emph{completion} equivalently as the usual completion of its separation, or as a colimit in $\Top_k$,
\[
\widehat{E} := \invlim_{W \in \setS(E)} E/W = \widehat{E_s}.
\]
It is straightforward to verify that the assignments $E \mapsto E_s$ and $E \mapsto \widehat{E}$ both naturally extend to endofunctors of $\Top_k$. See also~\cite[Section~1]{positselski:exact} for many more details.

Let $E$ and $F$ be linearly topologized vector spaces. Following~\cite[Section~12]{positselski:exact}, let $E \otimes^! F$ denote the vector space $E \otimes F$ equipped with the linear topology whose open subspaces are the subspaces $W \subseteq E \otimes F$ for which there exist open subspaces $E_0 \in \setS(E)$ and $F_0 \in \setS(F)$ such that 
\[
E_0 \otimes F + F \otimes E_0 \subseteq W.
\]
Note that the basic open sets described above are precisely the kernels of the natural surjections $E \otimes F \twoheadrightarrow (E/E_0) \otimes (F/F_0)$. It follows immediately that if $E$ and $F$ both have cofinite topologies, then the same is true for $E \otimes^! F$. In this way we obtain a monoidal category structure $(\CF_k, \otimes^!, k)$ on the category of cofinite linearly topologized spaces.
 
Now let $E \ctens F$ denote the completion of $E \otimes^! F$. While this construction is defined for general objects of $\Top_k$ that are not necessarly separated or complete, one can readily verify that
\[
E \ctens F \cong \invlim (E/E' \otimes F/F') \cong \widehat{E} \ctens \widehat{F},
\]
where $E'$ and $F'$ range over the open subspaces of $E$ and $F$, respectively. Furthermore, if $E$ and $F$ are pseudocompact, the equation above makes it clear that the same is true for $E \ctens F$, and we obtain the \emph{completed tensor product} defined in~\cite{brumer:pseudocompact}. There it was characterized by the following universal property:
the pseudocompact space $E \ctens F$ is equipped with a continuous bilinear map $B \colon E \times F \to E \ctens F$ (where $B(v,w) = v \ctens w$) such that every continuous bilinear map $E \times F \to X$ to a pseudocompact $k$-vector space $X$ factors uniquely through a continuous linear map $\phi$ as follows:
\[
E \times F \xrightarrow{B} E \ctens F \xrightarrow{\phi} X.
\]

We now turn to the relationship between the tensor products discussed above and dual space functors. 
Because $\Top_k$ is a $k$-linear category, we have a $k$-linear \emph{continuous dual} functor
\[
(-)^\circ := \Top_k(-,k) \colon \Top_k\op \to \Vect_k.
\]
The following key result describes how the continuous dual plays well with the topological tensor $-\otimes^! -$, especially when restricted to the category $\CF_k$ of spaces with cofinite topology. The definitions of lax and strong monoidal functors can be found in~\cite[XI.2]{maclane}, noting that MacLane uses the term ``monoidal'' instead of ``lax monoidal.''

\begin{theorem}\label{thm:tensor dual}
The continuous dual forms a lax monoidal functor
\[
(-)^\circ \colon (\Top_k, \otimes^!, k)\op \to (\Vect_k, \otimes, k),
\]
which restricts to a strong monoidal functor
\[
(-)^\circ \colon (\CF_k, \otimes^!, k)\op \to (\Vect_k, \otimes, k).
\]
\end{theorem}

\begin{proof}
First note that duality preserves the monoidal unit as $k^\circ = k^* \cong k$.
Let $E$ and $F$ be linearly topologized vector spaces. 
Consider the natural embedding for the discrete dual spaces $E^* \otimes F^* \hookrightarrow (E \otimes F)^*$, where a pure tensor $\phi \otimes \psi \in E^* \otimes F^*$ acts on a pure tensor $e \otimes f \in E \otimes F$ by
\[
(\phi \otimes \psi)(e \otimes f) = \phi(e) \psi(f).
\]
If $\phi \in E^\circ$ and $\psi \in F^\circ$ both have open kernels, note that 
\[
\ker \phi \otimes F + E \otimes \ker \psi \subseteq \ker(\phi \otimes \psi)
\]
implies that $\ker(\phi \otimes \psi)$ is open in $E \otimes^! F$, so that $\phi \otimes \psi \in (E \otimes^! F)^\circ$. Thus the natural embedding for discrete dual spaces restricts to an analogous embedding for continuous dual spaces:
\[
\xymatrix{
E^* \otimes F^* \ar@{^{(}->}[r] & (E \otimes F)^* \\
E^\circ \otimes F^\circ \ar@{^{(}->}[u] \ar@{^{(}->}[r]^{\Phi_{E,F}} & (E \otimes^! F)^\circ \ar@{^{(}->}[u]
}
\]
To show that this natural embedding
\begin{equation}\label{eq:embedding}
\Phi_{E,F} \colon E^\circ \otimes F^\circ \hookrightarrow (E \otimes^! F)^\circ
\end{equation}
makes $(-)^\circ$ into a lax monoidal functor, we must check the coherence axioms~\cite[XI.2(3,4)]{maclane}.
The left and right unitor diagrams are easily seen to commute:
\[
\xymatrix{
k \otimes E^\circ \ar[r]^\sim \ar[d]^[@!-90]{\sim} & E^\circ & E^\circ \otimes k \ar[r]^\sim \ar[d]^[@!-90]{\sim} & E^\circ \\
k^\circ \otimes E^\circ \ar@{^{(}->}[r] & (k \otimes^! E)^\circ \ar[u]_[@!-90]{\sim} & E^\circ \otimes k^\circ \ar@{^{(}->}[r] & (E \otimes^! k)^\circ \ar[u]_[@!-90]{\sim}
}
\]
The coherence diagram for the associator takes the following form:
\begin{equation}\label{eq:coherence}
\begin{split}
\xymatrix{
E^\circ \otimes (F^\circ \otimes G^\circ) \ar[d]_{\id_{E^\circ} \otimes \Phi_{F,G}} \ar[r]^{\alpha_1} & (E^\circ \otimes F^\circ) \otimes G^\circ \ar[d]^{\Phi_{E,F} \otimes \id_{G^\circ}} \\
E^\circ \otimes (F \otimes^! G)^\circ \ar[d]_{\Phi_{E, F \otimes^! G}} & (E \otimes^! F)^\circ \otimes G^\circ \ar[d]^{\Phi_{E \otimes^! F,G}} \\
(E \otimes^! (F \otimes^! G))^\circ \ar[r]^{\alpha_2^\circ} & ((E \otimes^! F) \otimes^! G)^\circ 
}
\end{split}
\end{equation}
where the morphisms $\alpha_1$ in $\Vect_k$ and $\alpha_2$ in $\Top_k$ are the appropriate associators. To verify that this diagram commutes, we may take $(\phi_1, \phi_2, \phi_3) \in E^\circ \times F^\circ \times G^\circ$ and chase the pure tensor $\phi_1 \otimes (\phi_2 \otimes \phi_3) \in E^\circ \otimes (F^\circ \otimes G^\circ)$ around the diagram. The composite morphism down the left column interprets $\phi_1 \otimes (\phi_2 \otimes \phi_3)$ as a functional on $E \otimes^! (F \otimes^! G)$ by the action
\[
(\phi_1 \otimes (\phi_2 \otimes \phi_3))(x \otimes (y \otimes z)) = \phi_1(x) \phi_2(y) \phi_3(z).
\]
The actions of $\alpha_1$ and $\alpha_2$ simply regroup the tensor factors, and the composite down the right column of the diagram acts in a completely analogous way. In this way one can check that the diagram commutes when restricted to pure tensors, so by linearity it commutes in $\Vect_k$.
Thus $(-)^\circ$ is a lax monoidal functor $(\Top_k, \otimes^!, k)\op \to (\Vect_k, \otimes, k)$. 

To show that $(-)^\circ$ forms a strict monoidal functor on $(\CF_k, \otimes^!, k)$, we assume that $E$ and $F$ have cofinite topologies, and we show that the embedding $\Phi_{E,F}$ is surjective. Thus let $\varphi \in (E \otimes^! F)^\circ$. Then there exist open subspaces $E_0 \subseteq E$ and $F_0 \subseteq F$ such that $E^0 \otimes F + E \otimes F_0 \subseteq \ker \varphi$. Thus $\varphi$ factors as
\[
E \otimes F \twoheadrightarrow E/E_0 \otimes F/F_0 \xrightarrow{\overline{\varphi}} k.
\]
Because $E$ and $F$ are cofinite, both $E_0$ and $F_0$ both have finite codimension, and we have
\[
\overline{\varphi} \in (E/E_0 \otimes F/F_0)^* \cong (E/E_0)^* \otimes (F/F_0)^*.
\]
So we may in fact write 
\[
\varphi = \sum_{i=1}^n \phi_i \otimes \psi_i \in E^* \otimes F^*
\]
 where each $\ker \phi_i \supseteq E_0$ and $\ker \psi_i \supseteq F_0$, making each $\phi_i \in E^\circ$ and $\psi_i \in F^\circ$. So in this case the map~\eqref{eq:embedding} is an isomorphism as desired.
\end{proof}

In the opposite direction, the discrete dual space gives a $k$-linear functor
\[
(-)^* := \Hom_k(-,k) \colon \Vect_k\op \to \PC_k
\]
where the dual $V^* = \Hom(V,k)$ of a vector space $V$ is equipped with the \emph{finite topology}: the linear topology whose basic open subspaces are the annihilators $X^\perp = \{\phi \in V^* \mid \phi(X) = 0\}$ of all finite-dimensional subspaces $X \subseteq V$. 
Restricting the continuous dual to the subcategory $\PC_k$ of $\Top_k$, it is well known that these two functors are mutually quasi-inverse, yielding a duality between $\Vect_k$ and $\PC_k$; see~\cite[IV.4]{gabriel} or~\cite[Proposition~2.6]{simson}. 

We will show in Corollary~\ref{cor:monoidal duality} below that this duality can be enhanced to a \emph{strong monoidal} duality of monoidal categories with appropriate tensor products. 
This will provide an alternative explanation for another well-known duality between coalgebras and pseudocompact algebras~\cite[Theorem~3.6(d)]{simson}.
Recall that a \emph{pseudocompact algebra} $A$ is a linearly topologized algebra that is separated, complete, and has a neighood basis of zero consisting of open ideals having finite codimension, or equivalently, if it is the topological inverse limit of finite-dimensional discrete algebras. We let $\PCAlg$ denote the category of pseudocompact $k$-algebras with continuous algebra homomorphisms.

Let $(C,\Delta,\epsilon)$ be a coalgebra. As is well known~\cite[Section~4.0]{sweedler:hopf}, the 
dual vector space $C^*$ is endowed with the structure of a $k$-algebra, with the 
\emph{convolution} product induced by restricting the dual of the comultiplication 
$\Delta^* \colon (C \otimes C)^* \to C^*$ to the subspace 
$C^* \otimes C^* \hookrightarrow (C \otimes C)^*$, and having unit $\epsilon \in C^*$.
Explicitly, the convolution product is given as follows: given $f,g \in C^*$ and an
element $q \in C$ with $\Delta(q) = \sum q_{(1)} \otimes q_{(2)}$ written in Sweedler notation, their
convolution $fg \in C^*$ is the element that acts via
\[
fg(q) = \sum f(q_{(1)}) g(q_{(2)}).
\]
Thanks to the fundamental theorem of coalgebras~\cite[Theorem~2.2.3]{radford:hopf}, $C$ is the directed union of its finite-dimensional subcoalgebras. Thus the dual algebra is an inverse limit $C^* \cong \varprojlim S^*$, where $S$ ranges over the finite-dimensional subcoalgebras $S \subseteq C$. By endowing each of the finite-dimensional algebras $S^*$ with the discrete topology, we may view the inverse limit above in the category of topological algebras. The resulting topology is the \emph{finite topology} on $C^*$. It is evident from the construction that $C^*$ is a pseudocompact $k$-algebra.

\begin{corollary}\label{cor:monoidal duality}
There is a strong monoidal duality between $(\Vect_k, \otimes, k)$ and $(\PC_k, \ctens, k)$ given by the dual space functors
\begin{align*}
(-)^\circ \colon \PC_k\op &\to \Vect_k, \\
(-)^* \colon \Vect_k\op &\to \PC.
\end{align*}
These functors induce further dualities between:
\begin{enumerate}
\item $\Coalg$ and $\cat{PCAlg}$, 
\item $C \lComod$ and $C^* \lPC$ for any coalgebra $C$.
\end{enumerate}
\end{corollary}

\begin{proof}
As mentioned before, it is well-established~\cite[IV.4]{gabriel} that the continuous and discrete dual functors provide a duality between $\PC_k$ and $\Vect_k$. These functors clearly interchange the monoidal units. Thus we only need to examine the effect of these functors on the tensor structures.

First let $E$ and $F$ be objects of $\PC_k$. Recalling that $E \ctens F$ is the completion of the linearly topologized space $E \otimes^! F$, we obtain a natural isomorphism
\begin{equation}\label{eq:completing}
(E \ctens F)^\circ = \Top_k (E \ctens F, k) \cong \Top_k(E \otimes^! F, k) = (E \otimes^! F)^\circ.
\end{equation}
Because the pseudocompact spaces $E$ and $F$ have cofinite topologies, Theorem~\ref{thm:tensor dual} yields a natural isomorphism
\[
(E \otimes^! F)^\circ \cong E^\circ \otimes F^\circ,
\]
which composes with~\eqref{eq:completing} to provide natural isomorphisms
\[
\Psi_{E,F} \colon E^\circ \otimes F^\circ \xrightarrow{\sim} (E \ctens F)^\circ.
\]
The unitor coherence axioms are as easily verified as before. The associator coherence diagram can be derived from~\eqref{eq:coherence} as follows. For pseudocompact spaces $E$, $F$, and $G$, one verifies that the induced diagrams
\[
\xymatrix{
E^\circ \otimes (F \ctens G)^\circ \ar[d]_{\Psi_{E,F \ctens G}} & E^\circ \otimes (F \otimes^! G)^\circ \ar[l]_-\sim  \ar[d]^{\Phi_{E, F \otimes^! G}} \\
(E \ctens (F \ctens G))^\circ  & (E \otimes^! (F \otimes^! G))^\circ \ar[l]_-\sim
}
\]
and
\[
\xymatrix{
(E \otimes^! F)^\circ \otimes G^\circ \ar[r]^-\sim  \ar[d]_{\Phi_{E \otimes^! F, G}} & (E \ctens F)^\circ \otimes G^\circ \ar[d]^{\Psi_{E \ctens F, G}} \\
 ((E \otimes^! F) \otimes^! G)^\circ \ar[r]^-\sim & ((E \ctens F) \ctens G)^\circ 
}
\]
commute, where the horizontal isomorphisms are induced as in~\eqref{eq:completing}. Then ``pasting'' these to~\eqref{eq:coherence} results in a larger commuting diagram, whose outer paths will be the desired coherence diagram for the isomorphisms of $\Psi$.
Thus the continuous dual $(-)^\circ$ is a strong monoidal functor.

Since the quasi-inverse $(-)^*$ must be adjoint to $(-)^\circ$, we can deduce~\cite[3.3]{schwedeshipley:monoidal} that it is also a strong monoidal functor. It is also possible to verify this directly. While we will not check the coherence axioms here, we will at least demonstrate how to locate the relevant natural isomorphisms.
Fix objects $V$ and $W$ of $\Vect_k$. Let $\{V_i\}$ be an indexing of the finite-dimensional subspaces of $V$, so that $V \cong \dirlim V_i$ in $\Vect_k$. Then we obtain natural isomorphisms
\begin{align*}
(V \otimes W)^* &= \Hom_k(V \otimes W,k) \\
&\cong \Hom_k(V, \Hom(W, k)) \\
&\cong \Hom_k(\dirlim V_i, W^*) \\
&\cong \invlim \Hom_k(V_i, W^*) \\
&\cong \invlim V_i^* \otimes W^* \\
&\cong V^* \ctens W^*,
\end{align*}
for which it would remain to verify coherence.

Finally, we derive the further dualities~(1) and~(2).
Note that $\Coalg$ is the category of comonoid objects in $\Vect_k$, while we verify in Lemma~\ref{lem:pseudocompact monoid} below that $\PCAlg$ is the category of monoid objects in $\PC_k$. Since the strong monoidal duality interchanges comonoids and monoids, we obtain the duality~(1) between coalgebras and pseudocompact algebras. 
Finally, fix a $k$-coalgebra $C$. Then $C \lComod$ is the category of left $C$-comodule objects over the comonoid $C$ in $\Vect$, while $C^* \lPC$ is the category of monoid objects over the monoid $C^*$ in $\PC_k$. So the duality~(2) follows similarly.
\end{proof} 

\begin{lemma}\label{lem:pseudocompact monoid}
Pseudocompact algebras are precisely the monoid objects in the monoidal category $(\PC_k, \ctens, k)$. 
\end{lemma}

\begin{proof}
If $A$ is a pseudocompact $k$-algebra, then its multiplication induces a continuous bilinear map $m \colon A \times A \to A$, which factors through the completed tensor product $\widehat{m} \colon A \ctens A \to A$ thanks to its universal property~\cite{brumer:pseudocompact}. It is then clear that this map and the usual unit map $u \colon k \to A$ form a monoid object $(A, \widehat{m}, u)$ in $(\PC_k, \ctens, k)$. 

Conversely, suppose that $(A, m, u)$ is a monoid object of $(\PC_k, \ctens, k)$. If we view the multiplication as a map 
\[
m^! \colon A \otimes^! A \hookrightarrow A \ctens A \xrightarrow{m} A,
\]
then in fact $(A, m^!, u)$ is a monoid object in $(\CF_k, \otimes^!, k)$. By Theorem~\ref{thm:tensor dual}, the continuous dual functor sends this to a comonoid, whose dual is then a pseudocompact algebra $(A^\circ)^*$. The natural isomorphism of pseudocompact spaces $A \cong (A^\circ)^*$ is then readily verified to be an isomorphism of topological algebras. Thus $A$ is in fact a pseudocompact algebra as desired.
\end{proof}

\subsection{The finite dual as a topological dual}\label{sub:finite dual}

We will now apply Theorem~\ref{thm:tensor dual} to show that the finite dual of an algebra can be viewed as the continuous dual with respect to an appropriate topology.
Let $A$ be a $k$-algebra with multiplication $m \colon A \otimes A \to A$.  Its \emph{finite dual coalgebra} has underlying vector space
\begin{align*}
A^\circ &= \{\phi \in A^* \mid m^*(\phi) \mbox{ lies in the subspace } A^* \otimes A^* \subseteq (A \otimes A)^*\} \\
&= \{\phi \in A^* \mid \ker \phi \mbox{ contains an ideal of finite codimension in } A\}.
\end{align*}
It happens that the restriction of $m^* \colon A^* \to (A \otimes A)^*$ to $A^\circ \subseteq A^*$ has image in $A^\circ \otimes A^\circ \subseteq A^* \otimes A^* \subseteq (A \otimes A)^*$. In this way $m^*$ restricts to a comultiplication 
\[
\Delta \colon A^\circ \to A^\circ \otimes A^\circ.
\]
This makes the finite dual into a coalgebra whose counit is dual to the unit of $A$. This can be alternatively described as a direct limit. If we set
\begin{equation}\label{eq:F(A)}
\F(A) = \{\mbox{ideals } I \unlhd A \mid \dim_k(A/I) < \infty\},
\end{equation}
then the finite dual coalgebra can alternatively be described as
\begin{equation}\label{eq:colimit}
A^\circ \cong \dirlim_{I \in \F(A)} (A/I)^*.
\end{equation}

If $A$ is a $k$-algebra, we define the \emph{cofinite} topology on $A$ as the linear topology whose open ideals are exactly the ideals of finite codimension---those ideals in the family $\F(A)$ defined in~\eqref{eq:F(A)}. 
A subspace of $A$ is open for this topology if and only if it contains an ideal of finite codimension.
This allows us to verify that the finite dual of $A$ and the continuous dual of $A$ are equal as subspaces of $A^*$:
\begin{align*}
\Top_k(A,k) &= \{\phi \in A^* \mid \ker \phi \supseteq E \mbox{ for some } E \in \setS(A)\} \\
&= \{\phi \in A^* \mid \ker \phi \supseteq I \mbox{ for some } I \in \F(A)\} \\
&= A^\circ.
\end{align*}
This is our justification for the notation $(-)^\circ = \Top_k(-,k)$.

\begin{theorem}\label{thm:continuous equals sweedler}
Let $A$ be a $k$-algebra, viewed as a linearly topologized algebra with its cofinite topology. Then its multiplication forms a continuous map between the topological vector spaces 
\[
m \colon A \otimes^! A \to A.
\]
Furthermore, the continuous dual functor $(-)^\circ = \Top_k(-,k)$ applied to the multiplication
yields the finite dual coalgebra $A^\circ$ with its comultiplication
\[
m^\circ \colon A^\circ \to (A \otimes^! A)^\circ \xrightarrow{\sim} A^\circ \otimes A^\circ.
\]
\end{theorem}

\begin{proof}
First we verify that multiplication yields a continuous map $m \colon A \otimes^! A \to A$. If $I \in \F(A)$ is an ideal of finite codimension, then $I \otimes A + A \otimes I$ is an open subspace of $A \otimes^! A$ such that $m(I \otimes A + A \otimes I) = IA + AI \subseteq I$. So $m$ is indeed continuous. 

It follows that $A$ with its unit $\eta \colon k \to A$ can be viewed as a monoid object $(A, m, \eta)$ in the monoidal category $(\CF_k, \otimes^!, k)$, or equivalently, as a comonoid object in the opposite tensor category. Because strong monoidal functors send comonids to comonoids, it follows from Theorem~\ref{thm:tensor dual} that $(A^\circ, m^\circ, \eta^\circ)$ is a coalgebra.

As explained above, the subspace $A^\circ \subseteq A^*$ is both the underlying vector space of the finite dual coalgebra and the continuous dual space. Let $\phi \in A^\circ$. The counit $\epsilon = \eta^\circ$ is easily seen to be given by $\epsilon(\phi) = \phi(1)$. Finally, the comultiplication $\Delta = m^\circ$ given by
\[
\Delta(\phi) = \phi \circ m \in (A \otimes^! A)^\circ \cong A^\circ \otimes A^\circ \subseteq A^* \otimes A^*
\]
agrees with the comultiplication of the finite dual (see~\cite[Section~6]{sweedler:hopf} or~\cite[Section~2.5]{radford:hopf}). This completes the proof.
\end{proof}

Note that every algebra homomorphism $f \colon A \to B$ is continuous with respect to the cofinite topologies on $A$ and $B$. Thus we see that the functoriality of finite dual $(-)^\circ \colon \Alg\op \to \Coalg$ can be explained in light of the duality of Theorem~\ref{thm:tensor dual} quite simply: the continuous dual functor sends monoid objects in $(\CF_k, \otimes^!, k)$ to comonoid objects in $(\Vect_k, \otimes, k)$.

\section{Dual coalgebras of twisted tensor products}\label{sec:twisted tensor}

Our goal in this section is to provide a method by which the finite dual can in principle be computed for a large class of algebras: those arising from a twisted tensor product construction. This makes the problem of computing dual coalgebras approachable for algebras that do not necessarily arise in connection with a Hopf algebra. 
Even so, the method has some novel applications in the context of bitwisted tensor products of Hopf algebras, as shown in Section~\ref{sec:applications}.

\subsection{Review of twisted and cotwisted tensor products}
\label{sub:crossed product}

A twisted tensor product of two $k$-algebras $A$ and $B$, as introduced in~\cite{csv:twisted}, is (informally) an algebra structure on the vector space $A \otimes B$ that is allowed to deform the usual tensor product algebra structure while retaining the respective algebra structures on $A$ and $B$. Similarly, ``smash coproducts and biproducts'' were introduced for coalgebras and bialgebras in~\cite{bd:cross, cimz:smash}. We recall the definitions and basic notions associated with these constructions below. 
(While we have slightly adjusted our terminology from some other sources in the literature, we have done so to maintain internal consistency of our language. We attempt to clarify these differences of terminology below.)

In the symmetric monoidal category $(\Vect_k, \otimes, k)$ of $k$-vector spaces, we will denote the ``tensor swap'' braiding for objects $V,W \in \Vect_k$ by
\begin{align*}
\sigma_{V,W} \colon V \otimes W &\to W \otimes V, \\
v \otimes w &\mapsto w \otimes v.
\end{align*}
Given a $k$-algebra $A$, we let $(A, m_A, \eta_A)$ denote its structure as a monoid object in $(\Vect_k, \otimes, k)$. Thus its multiplication is considered as a linear map $m_A \colon A \otimes A \to A$, and its unit map $\eta_A \colon k \to A$ is given by $\eta_A(1_k) = 1_A$.

Let $A$ and $B$ be algebras. Suppose that 
\[
\tau \colon B \otimes A \to A \otimes B
\]
is a linear map. Define a multiplication $m_\tau \colon (A \otimes B) \otimes (A \otimes B) \to A \otimes B$ by
\begin{equation}\label{eq:m rho}
m_\tau = (m_A \otimes m_B) \circ (\id_A \otimes \tau \otimes \id_B).
\end{equation}
If $m_\tau$ is associative with identity $1_A \otimes 1_B$, then the resulting algebra
\begin{equation}\label{eq:twisted tensor}
(A \otimes_\tau B, m_{A \otimes_\tau B}, \eta_{A \otimes_\tau B}) = (A \otimes B, m_\tau, \eta_A \otimes \eta_B)
\end{equation}
is a \emph{twisted tensor product} of $A$ and $B$, and $\tau$ is called a \emph{twisting map} for $A$ and $B$. In this case, the ``inclusion'' maps
\begin{align}\label{eq:inclusions}
\begin{split}
i_A &= \id_A \otimes 1_B \colon A \to A \otimes_\tau B \mbox{ and} \\
i_B &= 1_A \otimes \id_B  \colon B \to A \otimes_\tau B
\end{split}
\end{align} 
are algebra homomorphisms. 
The twisted tensor product can be characterized in terms of a universal property~\cite[Proposition~2.7]{csv:twisted} via $\tau$ and the inclusion maps above, but we will not recall that here since we do not make use of it.

Notice that the prototypical case where $\tau$ is chosen to be the tensor swap $\sigma := \sigma_{B,A}$ results in the usual tensor product algebra $A \otimes_\sigma B = A \otimes B$. In this way, we may view twisted tensor products as deformations of the tensor algebra.
On the other hand, if one takes an algebra $A$ with an automorphism $\sigma$ and a $\sigma$-derivation $\delta$, then setting $B = k[t]$ and choosing $\tau \colon k[t] \otimes A \to A \otimes k[t]$ appropriately, we can recover the Ore extension $A[t; \sigma, \delta] \cong A \otimes_\tau k[t]$ as in~\cite[Examples~2.11]{cimz:smash}.

A fundamental problem addressed in~\cites{csv:twisted, cimz:smash} is to isolate properties of a linear map $\tau \colon B \otimes A \to A \otimes B$ that make it a twisting map. The map $\tau$ is defined to be \emph{normal} if it satisfies the conditions
\begin{align}\label{eq:normal}
\begin{split}
\tau \circ (\eta_B \otimes \id_A) &= \id_A \otimes \eta_B, \\
\tau \circ (\id_B \otimes \eta_A) &= \eta_A \otimes \id_B.
\end{split}
\end{align}
(This is equivalent to saying that the restriction of $\tau$ to the subspaces $A \otimes 1_B$ and $1_A \otimes B$ agrees with the tensor swap $\sigma = \sigma_{B,A}$.)
Furthermore, $\tau$ is defined to be \emph{multiplicative} if it satisfies the conditions
\begin{align}\label{eq:multiplicative}
\begin{split}
\tau \circ (\id_B \otimes m_A) &= (m_A \otimes \id_B) \circ (\id_A \otimes \tau) \circ (\tau \otimes \id_A), \\
\tau \circ (m_B \otimes \id_A) &= (\id_A \otimes m_B) \circ (\tau \otimes \id_B) \circ (\id_B \otimes \tau).
\end{split}
\end{align}
These turn out to be necessary and sufficient conditions for $\tau$ to define a twisted tensor product.

\begin{proposition}\label{prop:normal multiplicative}
Let $A$ and $B$ be algebras. A linear map 
\[
\tau \colon B \otimes A \to A \otimes B
\] 
is a twisting map if and only if it is normal and multiplicative.
\end{proposition}

\begin{proof}
See~\cite[Proposition/Definition~2.3, Remark~2.4]{csv:twisted} and~\cite[Theorem~2.5]{cimz:smash}.
\end{proof}

Dually, let $C$ and $D$ be coalgebras with a linear map
\[
\phi \colon C \otimes D \to D \otimes C.
\]
Define a comultiplication $\Delta_\phi \colon (C \otimes D) \otimes (C \otimes D) \to C \otimes D$ by
\[
\Delta_\phi =(\id_C \otimes \, \phi \otimes \id_D) \circ (\Delta_C \otimes \Delta_D).
\]
If $\Delta_\phi$ is coassociative with counit $\eps_C \otimes \eps_D$, then the resulting coalgebra
\begin{equation}\label{eq:cross coalgebra}
(C \otimes^\phi D, \Delta_{C \otimes^\phi D}, \eps_{C \otimes^\phi D}) = (C \otimes D, \Delta_\phi, \eps_C \otimes \eps_D)
\end{equation}
will be called a \emph{cotwisted tensor product} of $A$ and $B$, and $\phi$ is a \emph{cotwisting map}. (This terminology differs from that of~\cite{cimz:smash, bct:cross}, where this construction was respectively called the \emph{smash coproduct} or the \emph{cross product coalgebra}. We have elected to use this terminology in order to minimize confusion with similarly named constructions in the literature.)

As one would expect, there is a characterization of cotwisting maps dual to that of Proposition~\ref{prop:normal multiplicative}. A linear map $\phi \colon C \otimes D \to D \otimes C$ is defined to be \emph{conormal} if it satisfies the equations
\begin{align}\label{eq:conormal}
\begin{split}
(\eps_D \otimes \id_C) \circ \phi &= \id_C \otimes \eps_D, \\
(\id_D \otimes \eps_C) \circ \phi &= \eps_C \otimes \id_D,
\end{split}
\end{align}
and $\phi$ is defined to be \emph{comultiplicative} if it satisfies the conditions
\begin{align}\label{eq:comultiplicative}
\begin{split}
(\id_D \otimes \Delta_C) \circ \phi &= (\phi \otimes \id_C) \circ (\id_C \otimes \, \phi) \circ (\Delta_C \otimes \id_D), \\
(\Delta_D \otimes \id_C) \circ \phi &= (\id_D \otimes \, \phi) \circ (\phi \otimes \id_D) \circ (\id_C \otimes \Delta_D).
\end{split}
\end{align}
These provide the following characterization of cotwisting maps.

\begin{proposition}[\cite{cimz:smash}*{Theorem~3.4}] \label{prop:conormal comultiplicative}
Let $C$ and $D$ be coalgebras. A linear map
\[
\phi \colon D \otimes C \to C \otimes D
\]
is a cotwisting map if and only if it is conormal and comultipliative.
\end{proposition}

Finally, suppose that $A$ and $B$ each have the structure of both an algebra and a coalgebra. Let $\tau \colon B \otimes A \to A \otimes B$ be a twisting map of the underlying algebras and let $\phi \colon A \otimes B \to B \otimes A$ be a cotwisting map of the underlying coalgebras. If the algebra $A \otimes_\tau B$ and coalgebra $A \otimes^\phi B$ structures together make $A \otimes B$ into a bialgebra, then we call this the \emph{bitwisted tensor product} bialgebra, denoted $A \otimes^\phi_\tau B$.
(Again, this construction has been given different names in the literature. It was defined as the \emph{smash biproduct} in~\cite{cimz:smash} and the \emph{crossed product bialgebra} in~\cite{bd:cross, bct:cross}. As before, we have chosen our terminology to avoid conflict with the often used ``cross product'' and ``smash product'' terminology.)

\subsection{Duality for twisted tensor products}

The formal duality between twisted tensor products and cotwisted tensor products suggests the naive idea that the finite dual of a twisted tensor product algebra might be a cotwisted tensor product coalgebra of the respective finite duals. 
Unfortunately, there are examples for which this fails; see Example~\ref{ex:not continuous} below for a specific instance. In order deduce such a result, we need appropriate assumptions in place on the twisting map. It turns out that a pertinent tool is the topological tensor product $- \otimes^! -$ discussed in Subsection~\ref{sub:monoidal} and its good behavior under the continuous dual functor $(-)^\circ \colon \CF_k \to \Top_k$. 
Recall that the cofinite topology on an algebra, defined before Theorem~\ref{thm:continuous equals sweedler}, is the linear topology whose open ideals are the ideals of finite codimension.

First we note that the topological tensor product has the following weak compatibility with twisted tensor products.

\begin{lemma}\label{lem:compatible topologies}
For algebras $A$ and $B$ and a twisting map $\tau \colon B \otimes A \to A \otimes B$, the linear isomorphism (given by the identity on underlying vector spaces)
\[
A \otimes^! B \to A \otimes_\tau B
\]
is continuous, where $A$, $B$, and $A \otimes_\tau B$ are endowed with their cofinite topologies.
\end{lemma}

\begin{proof}
Denote $S = A \otimes_\tau B$. Fix $K \in \F(S)$ and denote the canonical projection by $\pi \colon S \twoheadrightarrow S/K =: F$, where the codomain is finite-dimensional. Composing with the inclusion $i_A \colon A \to S$ of~\eqref{eq:inclusions} gives a homomorphism $A \to S \to F$ that factors through a finite-dimensional algebra $A \twoheadrightarrow A_1 \subseteq F$ with kernel $I \in \F(A)$. Similarly, the composition of $i_B$ with $\pi$ must factor through a finite-dimensional algebra $B \twoheadrightarrow B_1 \hookrightarrow F$ with kernel $J \in \F(B)$. Because there is a vector space isomorphism $S \cong i_A(A) \otimes i_B(B)$, it follows that multiplication in $F$ yields a linear surjection $A_1 \otimes B_1 \twoheadrightarrow F$. It follows that
\[
I \otimes B + A \otimes J \subseteq K,
\]
which implies that $K$ is open in $A \otimes^! B$ and proves that the linear isomorphism is continuous. 
\end{proof}

Given a twisting map $\tau \colon B \otimes A \to A \otimes B$, we will wish to apply the continuous dual functor to obtain a cotwisting map. But this will only be possible in the case where $\tau$ is continuous. This can be equivalently characterized in terms of a stronger compatibility between the topological tensor and twisted tensor products as follows.

\begin{proposition}
\label{prop:topological conditions}
Let $A$ and $B$ be algebras with $\tau \colon B \otimes A \to A \otimes B$ a twisting map. Consider the following conditions, where we equip $A$, $B$, and $A \otimes_\tau B$ with their cofinite topologies where appropriate:
\begin{enumerate}[label=\textnormal{(\roman*)}]
\item $\tau \colon B \otimes^! A \to A \otimes^! B$ is continuous.
\item The map $A \otimes^! B \to A \otimes_\tau B$ of Lemma~\ref{lem:compatible topologies} is a homeomorphism.
\item There exist neighborhood bases of zero $\{I_\alpha\} \subseteq \F(A)$ and $\{J_\beta\} \subseteq \F(B)$ for which $\tau(B \otimes I_\alpha) \subseteq I_\alpha \otimes B$ and $\tau(J_\beta \otimes A) \subseteq A \otimes J_\beta$.
\end{enumerate}
Then $\textnormal{(i)} \iff \textnormal{(ii)} \impliedby \textnormal{(iii)}$.
\end{proposition}

\begin{proof}
Note that in the arguments below, while we carefully distinguish between topologies in $A \otimes^! B$ and $A \otimes_\tau B$, many computations are made by identifying both of their underlying vector spaces with  $A \otimes B$. We trust that this will not cause confusion for the vigilant reader.

(ii)$\implies$(i): 
Fix a basic open neighorhood of zero in $\F(A \otimes^! B)$ of the form $N = I \otimes B + A \otimes J$ for $I \in \F(A)$ and $J \in \F(B)$. To prove continuity of $\tau$, we must produce $K \in \F(B \otimes^! A)$ such that $\tau(K) \subseteq N$.  By condition~(ii), there exits an open ideal $U \in \F(A \otimes_\tau B)$ such that $U \subseteq N$ as subspaces of $A \otimes B$, and by Lemma~\ref{lem:compatible topologies} there exists a basic open subspace of $A \otimes^! B$ contained in $U$. Altogether, this means that there exist $I_0 \in \F(A)$ and $J_0 \in \F(B)$ such that 
\[
I_0 \otimes B + A \otimes J_0 \subseteq U \subseteq I \otimes B + A \otimes J.
\]
Note that $K := B \otimes I_0 + J_0 \otimes A$ is an element of $\F(B \otimes^! A)$. We claim that $\tau(K) \subseteq N$ as required above. We compute by interpreting $\tau$ in terms of multiplication in $A \otimes_\tau B$ and liberally applying the linear identification of $A \otimes^! B$ with $A \otimes_\tau B$:
\begin{align*}
\tau(K) &= \tau(B \otimes I_0) + \tau(J_0 \otimes A) \\
&= (1 \otimes B)(I_0 \otimes 1) + (1 \otimes J_0)(A \otimes 1) \\
&\subseteq (1 \otimes B) U + U (A \otimes 1) \subseteq U,
\end{align*}
because $U$ is an ideal of $A \otimes_\tau B$. It follows that $\tau(K) \subseteq U \subseteq N$ as desired.

(i)$\implies$(ii): 
To establish~(ii), it suffices to show that the natural map $A \otimes^! B \to A \otimes_\tau B$ of Lemma~\ref{lem:compatible topologies} is open; because this map is linear, we need only test at neighborhoods of zero. So fix $I \in \F(A)$ and $J \in \F(B)$ which form a basic open neighborhood of zero 
\[
N = I \otimes B + A \otimes J \in \F(A \otimes^! B).
\] 
Let $\sigma = \sigma_{A,B} \colon A \otimes^! B \to B \otimes^! A$ denote the tensor swap, which is clearly continuous. Since $\tau$ is assumed to be continuous, it follows that $\theta := \tau \circ \sigma \colon A \otimes^! B \to A \otimes^! B$ is also continuous. 
Then $\theta^{-1}(N)$ is open in $A \otimes^! B$, so there exist $I_0 \in \F(A)$ and $J_0 \in \F(B)$ such that 
\[
N_0 := I_0 \otimes B + A \otimes J_0 \subseteq \theta^{-1}(N). 
\]
Let $U$ denote the ideal of $A \otimes_\tau B$ generated by the subspace $N_0$. Because $U$ contains the subspace $N_0$ of finite codimension, we also have $U \in \F(A \otimes_\tau B)$. We will show below that $U \subseteq N$.
This will imply that $N$ is open in $A \otimes_\tau B$ so that~(ii) will be established. 

Since $N_0 = I_0 \otimes B + A \otimes J_0$, to prove that the ideal $U$ generated by this subspace lies in the subspace $N$ it suffices to show that the ideals separately generated by $I_0 \otimes B$ and $A \otimes J_0$ both lie in $N$. It is straightforward to see that $I_0 \otimes B$ is a right ideal within $A \otimes_\tau B$. To verify that it is also a left ideal, we make note of the following facts:
\begin{itemize}
\item $I_0 \otimes B \subseteq N_0 \subseteq \theta^{-1}(N)$ implies that $\tau(B \otimes I_0) = \theta(I_0 \otimes B) \subseteq N$,
\item $N$ is an $(A,B)$-subbimodule of $A \otimes_\tau B$.
\end{itemize}
Thus we may compute using the product in $A \otimes_\tau B$ as follows:
\[
(A \otimes_\tau B)  (I_0 \otimes B) = A \tau(B \otimes I_0) B \subseteq A N B = N.
\]
So $I_0 \otimes B$ generates a subideal of $N$, and a symmetric argument shows that the same is true for $A \otimes J_0$. It follows that $U = (A \otimes_\tau B) N_0 (A \otimes_\tau B) \subseteq N$ so that~(ii) is established.

(iii)$\implies$(i): Fix a basic open neighborhood of zero $U = I \otimes B + A \otimes J$ in $A \otimes^! B$ with $I \in \F(A)$ and $J \in \F(B)$. Then there exist elements $I_\alpha$ and $J_\beta$ of the neighborhood bases in~(iii) such that $I_\alpha \subseteq I$ and $J_\beta \subseteq J$. Thus we may produce the basic open neighborhood $V = B \otimes I_\alpha + J_\beta \otimes A \in \F(B \otimes^! A)$ that satisfies
\begin{align*}
\tau(V) &= \tau(B \otimes I_\alpha) + \tau(J_\beta \otimes A) \\
&\subseteq I_\alpha \otimes B + A \otimes J_\beta \\
&\subseteq U.
\end{align*}
Because $\tau$ is linear, this suffices to prove that $\tau$ is continuous.
\end{proof}

The following example illustrates that twisting maps can easily fail to be continuous.

\begin{example}\label{ex:not continuous}
For the algebras $A = k[x]$ and $B = k[y]$, it is possible to choose twisting maps $\tau \colon B \otimes A \to A \otimes B$ such that $A \otimes_\tau B$ has no proper ideals of finite codimension. 
For instance, if $k$ has characteristic zero and we set
\[
\tau(y^m \otimes f) = \sum_{i=0}^m \binom{m}{i} \, \partial_x^{m-i} f \otimes y^i,
\]
then $A \otimes_\tau B = A_1(k)$ is the first Weyl algebra, which infamously has no nonzero finite-dimensional representations.
In such cases, the finite topology on $A \otimes_\tau B$ is the indiscrete topology, which radically differs from the topology on $A \otimes^! B$. It follows from Proposition~\ref{prop:topological conditions} that $\tau$ is not continuous. We also see that $(A \otimes_\tau B)^\circ = 0$ is not isomorphic to any cotwisted tensor product of the form $A^\circ \otimes^\phi B^\circ$, since such a coalgebra is nonzero by construction.
\end{example}

We now arrive at the major result of this section. It shows that continuity of the twisting map with respect to the cofinite topologies is sufficient to allow the finite dual of a twisted tensor product to be a cotwisted tensor product of the expected form.

\begin{theorem}\label{thm:twisted dual}
Let $A$ and $B$ be $k$-algebras with a twisting map $\tau$. If $\tau \colon B \otimes^! A \to A \otimes^! B$ is continuous, where $A$ and $B$ are endowed with their cofinite topologies, then the continuous dual 
\[
\tau^\circ = \Top_k(\tau,k) \colon A^\circ \otimes B^\circ \to B^\circ \otimes A^\circ
\]
is a cotwisting map and the continuous dual of the topological isomorphism $A \otimes^! B \to A \otimes_\tau B$ yields an isomorphism of coalgebras
\[
(A \otimes_\tau B)^\circ \xrightarrow{\sim} A^\circ \otimes^{\tau^\circ} B^\circ.
\]
\end{theorem}

\begin{proof}
Because $\tau$ is continuous, its continuous dual $\tau^\circ = \Top_k,(\tau,k)$ is defined. Recall that the topologies on $A$ and $B$ are cofinite. Then by Theorem~\ref{thm:tensor dual}, we can view this as a map
\[
\tau^\circ \colon A^\circ \otimes B^\circ \xrightarrow{\sim}  (A \otimes^! B)^\circ \to (B \otimes^! A)^\circ \xrightarrow{\sim} B^\circ \otimes A^\circ
\]
By Proposition~\ref{prop:normal multiplicative}, $\tau$ is normal and multiplicative. Theorem~\ref{thm:tensor dual} now implies that $\tau^\circ$ is conormal and comultiplicative, since these are formally dual properties. 
Thus it follows from Proposition~\ref{prop:conormal comultiplicative} that $\tau^\circ$ is a cotwisting map for $A^\circ$ and $B^\circ$.

Denote $S = A \otimes_\tau B$ and $C = A^\circ \otimes^{\tau^\circ} B^\circ$.
By Theorem~\ref{thm:continuous equals sweedler}, the finite dual of $S$ and its comultiplication respectively coincide with the continuous dual of $S$ and its multiplication. By Proposition~\ref{prop:topological conditions} the linear isomorphism 
\[
\Phi \colon A \otimes^! B \to S
\] 
is a homeomorphism. Thus the continuous dual functor yields a linear isomorphism 
\begin{equation}\label{eq:dual coalgebra}
S^\circ \xrightarrow{\sim} (A \otimes^! B)^\circ \cong A^\circ \otimes B^\circ = C, 
\end{equation}
where we identify the vector space $A^\circ \otimes B^\circ$ with the cotwisted tensor product coalgebra $C$.
It only remains to show that this is a morphism of coalgebras. 

To achieve this goal, we will treat the isomorphism~\eqref{eq:dual coalgebra} as the identity map. (This is reasonable because it is defined in terms of the linear homeomorphism $\Phi$ above, which the identity on the underlying vector space.) We must then show that the comultiplication of $S^\circ$ and $C$ coincide. 
By definition of the twisted tensor product $S = A \otimes_\tau B$, we have 
\[
m_S = (m_A \otimes m_B) \circ (\id_A \otimes \tau \otimes \id_B).
\]
Applying the continuous dual functor to the above formula and invoking Theorem~\ref{thm:tensor dual} in the third equality below yields
\begin{align*}
\Delta_{S^\circ} &= (m_S)^\circ \\
&= (\id_A \otimes \tau \otimes \id_B)^\circ \circ (m_A \otimes m_B)^\circ \\
&= (\id_A^\circ \otimes \tau^\circ \otimes \id_B^\circ) \circ (m_A^\circ \otimes m_B^\circ) \\
&= (\id_{A^\circ} \otimes \tau^\circ \otimes \id_{B^\circ} ) \circ (\Delta_{A^\circ} \otimes \Delta_{B^\circ}) \\
&= \Delta_C
\end{align*}
as desired. An easier computation similarly shows that the unit $\eta_S = \eta_A \otimes \eta_B$ dualizes to the counit $\eps_C = \eps_{A^\circ} \otimes \eps_{B^\circ}$. 
\end{proof}

\begin{example}
For algebras $A$ and $B$, the tensor product algebra has dual coalgebra given by
\[
(A \otimes B)^\circ \cong A^\circ \otimes B^\circ.
\]
While this is straightforward to prove from~\eqref{eq:colimit}, it also follows immediately from Theorem~\ref{thm:twisted dual}. This is the special case where $\tau = \sigma_{B,A}$ is the ``tensor swap'' map, which is evidently continuous.
\end{example}

\section{Applications of twisted tensor product duality}\label{sec:applications}

In this section, we specialize Theorem~\ref{thm:twisted dual} to a few situations of particular interest. This includes Ore extensions, smash products with Hopf algebras, and bitwisted tensor product bialgebras.

In order to apply Theorem~\ref{thm:twisted dual} in any particular situation, we must know that the twisting map $\tau \colon B \otimes A \to A \otimes B$ is continuous.
This raises the important question of how to recognize when a given twisting map is continuous. Our next goal will be to provide a sufficient condition for continuity of $\tau$ in Theorem~\ref{thm:centralize} below, amounting to the existence of large subalgebras of $A$ and $B$ that respectively centralize $B$ and $A$. 

We will say that an extension of rings $R_0 \subseteq R$ is \emph{finite} if $R$ is finitely generated as both a left $R_0$-module and a right $R_0$-module. 

\begin{lemma}\label{lem:finite extension}
Suppose that $A_0 \subseteq A$ is a finite extension of $k$-algebras. 
Then an ideal $I \unlhd A$ satisfies $I \in \F(A)$ if and only if there exists an ideal $I_0 \in \F(A_0)$ such that $I_0 \subseteq I$.
\end{lemma}

\begin{proof}
If $I \in \F(A)$ then we may set $I_0 = A \cap A_0$. The embedding $A_0/I_0 = A_0/(A_0 \cap I) \hookrightarrow A/I$ ensures that $I_0 \in \F(A)$. 

To establish the converse, note that if $I_0 \in \F(A)$ with $I_0 \subseteq I$, then the ideal $I' = AI_0A$ obtained by extending $I_0$ to $A$ satisfies $I' \subseteq I$, and if $I'$ has finite codimension then the same is true for $I$. Replacing $I$ with $I'$, we may thus reduce to the case where $I = A I_0 A$ for some $I_0 \in \F(A_0)$. 
Because $A$ is finite over $A_0$, we may fix $x_1,\dots x_m, y_1, \dots, y_n \in A$ with $A = \sum x_i A_0 = \sum A_0 y_j$. Then
\begin{align*}
A/I &= A/(AI_0A)  \\
&\cong A \otimes_{A_0} (A_0/I_0) \otimes_{A_0} A \\
&= \sum_{i,j} x_i \otimes_{A_0} (A_0/I_0) \otimes_{A_0} y_j.
\end{align*}
Since $A_0/I_0$ is finite-dimensional, we see that the same will be true for $A/I$. So $I \in \F(A)$ as desired. 
\end{proof}

\begin{theorem}\label{thm:centralize}
Let $\tau \colon B \otimes A \to A \otimes B$ be a twisting map. Suppose that there exist subalgebras $A_0 \subseteq A$ and $B_0 \subseteq B$ satisfying the following properties:
\begin{enumerate}[label=\textnormal{(\roman*)}]
\item The extensions $A_0 \subseteq A$ and $B_0 \subseteq B$ are both finite. 
\item The restriction of $\tau$ to each of the subspaces $B \otimes A_0$ and $B_0 \otimes A$ agrees with corresponding restriction of the tensor swap $\sigma_{B,A} \colon B \otimes A \to A \otimes B$.
\end{enumerate}
Then $\tau$ is continuous, so that $(A \otimes_\tau B)^\circ \cong A^\circ \otimes^{\tau^\circ} B^\circ$. 
\end{theorem}

\begin{proof}
It is enough to verify that $\tau$ satisfies the condition of Proposition~\ref{prop:topological conditions}(iii). 
By Lemma~\ref{lem:finite extension}, the family $\{I_\alpha\} \subseteq \F(A)$ of ideals in $A$ that are of the form $I_\alpha = A I_0 A$ for some $I_0 \in \F(A_0)$ is a neighborhood basis of zero in $A$.
We claim that every such ideal $I_\alpha = A I_0 A$ satisfies $\tau(B \otimes I_\alpha) \subseteq I_\alpha \otimes B$. First note that the multiplicative property~\eqref{eq:multiplicative} of $\tau$ yields
\begin{align*}
\tau(B \otimes AI_0) &= \tau \circ (\id_B \otimes m_A)(B \otimes A \otimes I_0) \\
&= (m_A \otimes \id_B) \circ (\id_A \otimes \tau) \circ (\tau \otimes \id_A) (B \otimes A \otimes I_0) \\
&= (m_A \otimes \id_B) \circ (\id_A \otimes \tau) (\tau(B \otimes A) \otimes I_0) \\ 
&\subseteq (m_A \otimes \id_B) \circ (\id_A \otimes \tau) (A \otimes B \otimes I_0) \\ 
&= (m_A \otimes \id_B) (A \otimes \tau(B \otimes I_0)) \\
&= (m_A \otimes \id_B) (A \otimes I_0 \otimes B) \\
&= AI_0 \otimes B.
\end{align*}
It then follows from another application of~\eqref{eq:multiplicative} again that
\begin{align*}
\tau(B \otimes I_\alpha) &= \tau(B \otimes (A I_0) A) \\
&= \tau \circ (\id_B \otimes m_A) (B \otimes AI_0 \otimes A) \\
&= (m_A \otimes \id_B) \circ (\id_A \otimes \tau) \circ (\tau \otimes \id_A) (B \otimes A I_0 \otimes A) \\
&= (m_A \otimes \id_B) \circ (\id_A \otimes \tau) (\tau(B \otimes A I_0) \otimes A) \\
&\subseteq (m_A \otimes \id_B) \circ (\id_A \otimes \tau) (A I_0 \otimes B \otimes A) \\
&= (m_A \otimes \id_B) (A I_0 \otimes \tau(B \otimes A)) \\
&\subseteq (m_A \otimes \id_B) (A I_0 \otimes A \otimes B) \\
&= A I_0 A \otimes B = I_\alpha \otimes B
\end{align*}
as desired.

We may similarly define a neighborhood basis of zero $\{J_\beta\} \subseteq \F(B)$ consisting of those ideals of the form $J_\beta = B J_0 B$ for some $J_0 \in \F(B_0)$. A symmetric argument will verify that these ideals satisfy $\tau(J_\beta \otimes A) \subseteq A \otimes J_\beta$. We now conclude from Proposition~\ref{prop:topological conditions} that $\tau$ is continuous, and the isomorphism of dual coalgebras follows from Theorem~\ref{thm:twisted dual}.
\end{proof}

We remark that condition~(ii) above, which serves to ``tame'' the behavior of a twisting map, is reminiscent of some other conditions that have been used to deduce good properties of twisting maps in the literature. For instance, \emph{strongly graded twists} (in the terminology of~\cite{connergoetz:property}) preserve the Koszul property~\cite{jlps, waltonwitherspoon}, allow for computation of Gerstenhaber brackets on Hochschild cohomology~\cite{kmoow}, and have been used to construct noncommutative graded isolated singularities~\cite{heueyama}.
Similarly, a condition introduced in~\cite[(3.3)]{wangshen:twisted} (relative to a generating set of an algebra) was used to deduce Artin-Schelter regularity of twisted tensor products and in~\cite{shenzhoulu:nakayama} to understand their Nakayama automorphisms.

\subsection{Application to Ore extensions}

Ore extensions are one of the most fundamental methods of constructing
noncommutative algebras. The above criterion applies nicely in the case of an Ore extension by a finite order automorphism. Interpreted geometrically, it verifies the intuition that an Ore extension $A[t;\theta, \delta]$ is like a ring of functions on a ``twisted product'' of the affine line and the noncommutative space corresponding to $A$. 

In order to describe the affine line in terms of a coalgebra,
we will recall the terminology of coalgebras of distributions on $k$-schemes and some relevant facts  from~\cite[Subsection~2.3]{reyes:qmax}. If $X$ is a scheme over $k$, the \emph{coalgebra of distributions} on $X$ is the direct limit  
\[
\Dist(X) = \dirlim_S \Gamma(S, \O_S)^*,
\]
where $S$ ranges over all closed subschemes of $X$ that are finite over $k$. If $X$ is affine then by~\cite[Proposition~2.13]{reyes:qmax} we have
\begin{equation}\label{eq:affine distributions}
\Dist(X) \cong \Gamma(X, \O_X)^\circ.
\end{equation}
If $k$ is algebraically closed and $X$ is of finite type over~$k$, then the grouplike elements of $\Dist(X)$ correspond to the closed points of $X$. In particular, if $A$ is a commutative affine $k$-algebra then the grouplike elements of $A^\circ \cong \Dist(\Spec A)$ are in bijection with the maximal spectrum of $A$.
The commutative algebra $k[t]$ is the coordinate ring of the affine line over~$k$, so~\eqref{eq:affine distributions} specializes to
\[
k[t]^\circ \cong \Dist(\A^1_k).
\]
The coalgebra of distributions on the affine line was described geometrically in~\cite[Example~2.20]{reyes:qmax}.

Our result applies to Ore extensions $A[t; \theta, \delta]$ where $\theta$ has finite order and $\delta$ satisfies a sutiably strong nilpotence condition. To phrase this condition, we make the following definition. In the free monoid $\{u,v\}^\star$ (written in Kleene star notation) on an alphabet of two letters $u$ and $v$, let $W_i^m$ denote the words of length $m$ with $i$ occurences of $u$. For a word $w \in \{u,v\}^\star$ and linear endomorphisms $\theta$ and $\delta$ of $A$, we let $w_{\theta, \delta} \in \End_k(A)$ be the endomorphism that is the image of $w$ after freely extending $u \mapsto \theta$ and $v \mapsto \delta$ to a monoid homomorphism $\{u,v\}^\star \to \End_k(A)$. (That is, $w_{\theta, \delta}$ is $w$ ``evaluated at $u = \theta$ and $v = \delta$.'')

Note that the Ore extension is a twisted tensor product $A[t; \theta] = A \otimes_\tau B$ defined by a twisting map $\tau \colon k[t] \otimes A \to A \otimes k[t]$. Of course, $\tau$ is determined by the condition $\tau(t \otimes a) = \sigma(a) \otimes t + \delta(a) \otimes 1$. Applying this inductively, one can verify that $\tau$ is defined generally, for $a \in A$, by:
\begin{equation}\label{eq:Ore twist}
\tau(t^m \otimes a) =  \sum_{i=0}^m \sum_{w \in W_i^m} w_{\theta, \delta}(a) \otimes t^i.
\end{equation}

An immediate consequence of the above formula is that $t^d \in k[t]$ centralizes $A$ in $A[t; \theta, \delta]$ if and only if the following condition holds:
\begin{equation}\label{eq:central power}
\theta^d = \id_A \mbox{ and } \sum_{w \in W_i^d} w_{\theta, \delta} = 0 \mbox{ for } 0 \leq i < d.
\end{equation}
The above implies that $\delta^d = 0$ so that $\delta$ is nilpotent. In fact, two special cases of~\eqref{eq:central power} are:
\begin{itemize}
\item $\theta^d = \id_A$ and $\delta = 0$;
\item $\theta = \id_A$ and $\delta^d = 0$.
\end{itemize}

\begin{corollary}\label{cor:Ore}
Let $A$ be an algebra 
with an automorphism $\theta$ and a left $\theta$-derivation $\delta$, both of which are $k$-linear. Let $A^\theta \subseteq A$ denote the fixed subalgebra. Assume that~\eqref{eq:central power} holds and that
$A$ is finite over the subring $A^\theta \cap \ker \delta$.
Then the finite dual of the Ore extension $A[t; \theta, \delta]$ is isomorphic to a cotwisted tensor product coalgebra
\[
A[t; \theta, \delta]^\circ \cong A^\circ \otimes^\phi \Dist(\A^1_k)
\]
for a suitable cotwisting map $\phi$.
\end{corollary}

\begin{proof} 
Set $B = k[t]$, so that $A[t; \theta, \delta] = A \otimes_\tau B$ as in~\eqref{eq:Ore twist}.
We will show that $\tau$ satisfies the condition of Theorem~\ref{thm:centralize}. Let $A_0 = A^\theta \cap \ker \delta \subseteq A$, so that $A$ is finite over $A_0$ by hypothesis. Then certainly $A_0$ centralizes $t \in B$ within the Ore extension $A[t; \theta]$. It follows that the restriction of $\tau$ to $B \otimes A_0$ coincides with $\sigma_{B,A_0}$.

Next, let $B_0 = k[t^d] \subseteq k[t] = B$. Certainly $B$ is generated as a $B_0$-module by the finite set $\{1,t,\dots,t^{d-1}\}$. Condition~\eqref{eq:central power} means that $t^d$ centralizes $A$ in $A[t; \theta, \delta]$. It follows that the restriction of $\tau$ to $B_0 \otimes A$ coincides with $\sigma_{B_0, A}$. 

It now follows from Theorem~\ref{thm:centralize} that $\tau$ is continuous. So Theorem~\ref{thm:twisted dual} yields the isomorphism
\[
A[t; \theta, \delta]^\circ = (A \otimes_\tau B)^\circ \cong A^\circ \otimes^{\tau^\circ} B^\circ.
\]
The conclusion now follows from the fact that $B^\circ = k[t]^\circ \cong \Dist(\A^1_k)$ as in~\eqref{eq:affine distributions}, where $\phi$ corresponds to $\tau^\circ$ under this isomorphism.
\end{proof}

\subsection{Application to smash products and bitwisted tensor products}

We now turn our attention to smash product algebras.
Suppose $H$ a Hopf algebra and $A$ is a left $H$-module algebra. The smash product is a particular case of a twisted tensor product
\[
A \# H = A \otimes_\tau H
\]
where the linear map $\tau \colon H \otimes A \to A \otimes H$ is given in Sweedler notation by 
\begin{equation}\label{eq:smash twist}
\tau(h \otimes a) =  \sum h_{(1)}(a) \otimes h_{(2)}.
\end{equation}
It is well known~\cite[Chapter~4]{montgomery:hopf} that this endows $A \# H$ with the structure of an algebra, so that $\tau$ is a twisting map.
If we write the $H$-module product as $\lambda \colon H \otimes A \to A$, then the twisting map can be written as the composite 
\[
\tau \colon H \otimes A \xrightarrow{\Delta \otimes \id_A} H \otimes H \otimes A \xrightarrow{\id_H \otimes \sigma} H \otimes A \otimes H
\xrightarrow{\lambda \otimes \id_H} A \otimes H.
\]
Note that if $\lambda \colon H \otimes^! A \to A$ is continuous with respect to the finite topologies, then the composite map $\tau$ above is continuous. Indeed, the comultiplication $\Delta$ is an algebra homomorphism and thus is continuous, and it follows then that the maps in the above will all be continuous with respect to the finite topologies and topological tensor product. Thus if the module action $\lambda$ is continuous, we have an isomorphism of coalgebras
\[
(A \# H)^\circ \cong A^\circ \otimes^{\tau^\circ} H^\circ.
\]
A general description of the finite dual of a smash product (without restriction on $\lambda$) is given in~\cite[Proposition~11.4.2]{radford:hopf}, which requires a more complicated subcoalgebra $A^{\underline{\circ}} \subseteq A^\circ$ to achieve a similar decomposition. The isomorphism above shows that such technicalities can be avoided if $\lambda$ is sufficiently well-behaved.

In practice, one might wish for more straightforward conditions on the $H$-action that can be verified instead of topological continuity. The next result provides more familiar algebraic properties of the action that are sufficient.
Recall~\cite[Definition~1.7.1]{montgomery:hopf} that the subalgebra of \emph{invariants} of $H$ is
\[
A^H = \{a \in A \mid h(a) = \eps(h)a \mbox{ for all } h \in H\}.
\]
For instance, if $H = kG$ is a group algebra, so that $G$ acts by automorphisms on $A$, then $A^H = A^G$ is the usual subalgebra of $G$-invariants.

\begin{theorem}
Let $H$ be a Hopf algebra and let $A$ be a left $H$-module algebra. Suppose that the following hold:
\begin{enumerate}
\item $A$ is finite over the subalgebra $A^H$ of $H$-invariants;
\item the action of $H$ on $A$ factors through a finite-dimensional Hopf algebra.
\end{enumerate}
Then the twisting map $\tau$ of~\eqref{eq:smash twist} is continuous and there is an isomorphism of coalgebras $(A \# H)^\circ \cong A^\circ \otimes^{\tau^\circ} H^\circ$.
\end{theorem}

\begin{proof}
It follows by the construction of $H$-invariants that the restriction of the twisting map~\eqref{eq:smash twist} to the subspace $H \otimes A^H$ agrees with $\sigma_{H, A^H}$, the ``tensor swap'' map: for $h \in H$ and $a \in A^H$,
\begin{align*}
\tau(h \otimes a) &= \sum h_{(1)}(a) \otimes h_{(2)} \\
&= \sum \epsilon(h_{(1)}) a \otimes h_{(2)} \\
&= a \otimes \left( \sum \, \epsilon(h_{(1)}) h_{(2)} \right) \\
&= a \otimes h.
\end{align*}

Now let $K$ be a finite-dimensional Hopf algebra through which the action of $H$ on $A$ factors; more precisely, there is a Hopf algebra surjection $\pi \colon H \twoheadrightarrow K$ and a $K$-module algebra action $\overline{\lambda} \colon K \otimes A \to A$ so that the action $\lambda$ of $H$ on $A$ factors as 
\[
\lambda \colon H \otimes A \xrightarrow{\pi \otimes \id_A} K \otimes A \xrightarrow{\overline{\lambda}} A.
\]
The surjection $\pi$ makes $H$ into a left $K$-comodule via the coaction 
\[
H \xrightarrow{\Delta} H \otimes H \xrightarrow{\pi \otimes \id_H} K \otimes H,
\] 
which we denote by $\alpha$. 
This data relates to the twisting map of $A \# H$ through the following commuting diagram, where the composite across the top row is equal to $\tau$ and the vertical arrows are induced by $\pi$:
\[
\xymatrix{
H \otimes A \ar[r]^-{\Delta \otimes \id_A} \ar[dr]_{\alpha \otimes \id_A} & H \otimes H \otimes A \ar[r]^-{\id_H \otimes \sigma} \ar@{->>}[d] & H \otimes A \otimes H \ar[r]^-{\lambda \otimes \id_H} \ar@{->>}[d] & A \otimes H \\
 & K \otimes H \otimes A \ar[r]^{\id_K \otimes \sigma} & K \otimes A \otimes H \ar[ur]_{\overline{\lambda} \otimes \id_H} & 
}
\]

Denote the subalgebra of $K$-coinvariants~\cite[Definition~1.7.1]{montgomery:hopf} in $H$ by
\[
H_0 =  \{h \in H \mid \alpha(h) = 1_K \otimes h\}.
\]
Then for any $h \in H_0$ and $a \in A$, an examination of the commuting diagram above reveals that 
\[
\tau(h \otimes a) = 1_K(a) \otimes h = a \otimes h,
\]
so that the restriction of $\tau$ to $H_0 \otimes A$ is equal to $\sigma_{H_0, A}$.
Furthermore, $K$ being finite-dimensional implies~\cite[Theorem~8.2.4]{montgomery:hopf} that $H$ is a $K$-Galois extension of $H_0$, from which we can conclude~\cite[Theorem~1.7 and Corollary~1.8]{kreimertakeuchi} that $H$ is finitely generated (and projective) as both a left and right module over $H_0$.

Thus we have produced subalgebras $A_0 \subseteq A$ and $H_0 \subseteq H$ satisfying the hypotheses of Theorem~\ref{thm:centralize}, from which the desired conclusions follow.
\end{proof}

Historically speaking, the finite dual has mainly been of interest in the the study of Hopf algebras and bialgebras, as mentioned in Section~\ref{sec:intro}. Thus as a final application of these methods we provide sufficient conditions for the dual of a bitwisted tensor product product of two bialgebras algebra to be the bitwisted tensor product of the respective dual bialgebras in Corollary~\ref{cor:crossed bialgebra} below. We thank Hongdi Huang for an insightful question that inspired the result.

\begin{remark}\label{rem:Hopf dual}
It is a well-known fact~\cite[Theorem~9.1.3]{montgomery:hopf} that the finite dual of a bialgebra (resp., Hopf algebra) again has the structure of a bialgebra (resp., Hopf algebra). 
This can be interpreted in terms of topological duality as follows.
Let $H$ be a bialgebra with multiplication $m \colon H \otimes H \to H$ and comultiplication $\Delta \colon H \to H \otimes H$. Endow $H$ with the cofinite topology. We already know from Theorem~\ref{thm:continuous equals sweedler} that $m \colon H \otimes^! H \to H$ is continuous, and its continuous dual yields the comultiplication of $H^\circ$. Because $H$ is a bialgebra, the comultiplication is an algebra homomorphism and therefore is continuous as a map $\Delta \colon H \to H \otimes^! H$. The strong monoidal functor $(-)^\circ \colon \CF_k\op \to \Top_k$ sends comonoids to monoids, so the continuous dual of the comultiplication
\[
\Delta^\circ \colon H^\circ \otimes H^\circ \cong (H \otimes^! H)^\circ \to H^\circ,
\]
is a multiplication. Similarly, the unit and counit of $H$ have continuous duals that provide (co)units for the dual (co)multiplication structures. In this way $H^\circ$ becomes both a coalgebra and an algebra. Because the bialgebra axioms are self-dual (see the diagrams in the proof of \cite[Proposition~3.1.1]{sweedler:hopf}), they pass by continuous duality to $\Delta^\circ$ and $m^\circ$ in in order to show that $H^\circ$ becomes a bialgebra under these structures.
(If $H$ is a Hopf algebra, then its antipode considered as a map $H \to H\op$ is an algebra homomorphism and thus is continuous. So $S \colon H \to H$ is also continuous, allowing us to define $S^\circ \colon H^\circ \to H^\circ$. Again by self-duality of the axioms, this will be an antipode for $H^\circ$, making it into a Hopf algebra.)
\end{remark}

\begin{corollary}\label{cor:crossed bialgebra}
Let $A$ and $B$ be bialgebras, and suppose that $H = A \otimes^\phi_\tau B$ is a bitwisted tensor product. If the twisting and cotwisting maps are continuous 
\begin{align*}
\tau \colon B \otimes^! A &\to A \otimes^! B,\\
\phi \colon A \otimes^! B &\to B \otimes^! A,
\end{align*}
where $A$ and $B$ are equipped with their cofinite topologies, then there is an isomorphism of bialgebras 
\[
H^\circ \cong A^\circ \otimes^{\tau^\circ}_{\phi^\circ} B^\circ.
\]
\end{corollary}

\begin{proof}
As in Remark~\ref{rem:Hopf dual} above, $H^\circ$ is a bialgebra with multiplication $\Delta^\circ$ and comultiplication $m^\circ$. We know from Theorem~\ref{thm:twisted dual} that as a coalgebra we have $H^\circ \cong A^\circ \otimes^{\tau^\circ} B^\circ$. 
Also, because the comultiplication is given by
\[
\Delta_H = (\id_A \otimes \phi \otimes \id_B) \circ (\Delta_A \otimes \Delta_B),
\]
we are similarly able to apply Theorem~\ref{thm:tensor dual} to show that $\Delta_{H^\circ}$ coincides with the multiplication $m_{\phi^\circ}$ of $A^\circ \otimes_{\phi^\circ} B^\circ$. Indeed, with suitable harmless identifications as in the proof of Theorem~\ref{thm:twisted dual}, we have
\begin{align*}
m_{H^\circ} &= (\Delta_H)^\circ \\
&= ((\id_A \otimes \phi \otimes \id_B) \circ (\Delta_A \otimes \Delta_B))^\circ \\
&= (\Delta_A \otimes \Delta_B)^\circ \circ (\id_A \otimes \phi \otimes \id_B)^\circ \\
&= (\Delta_A^\circ \otimes \Delta_B^\circ) \circ (\id_A^\circ \otimes \phi^\circ \otimes \id_B^\circ) \\
&= (m_{A^\circ} \otimes m_{B^\circ}) \circ (\id_{A^\circ} \otimes \phi^\circ \otimes \id_{B^\circ}) \\
&= m_{\phi^\circ}.
\end{align*}
It now follows that $H^\circ \cong A^\circ \otimes^{\tau^\circ}_{\phi^\circ} B^\circ$ is a bitwisted tensor product. 
\end{proof}

\section{Examples of some noncommutative planes}
\label{sec:examples}

In this final section, we examine some specific cases where Theorem~\ref{thm:twisted dual} can be applied to determine the structure of the finite dual coalgebra of a twisted tensor product. We will focus some of the simplest twisted tensor products, which have the form 
\[
A \otimes_\tau B = k[x] \otimes_\tau k[y]
\] 
where $A = k[x]$ and $B = k[y]$ are polynomial algebras in a single indeterminate, 
such that the hypotheses of Theorem~\ref{thm:centralize} are satisfied.
Specifically, for each the examples considered below there is a positive integer $\ell$ such that 
\begin{equation}\label{eq:ellth powers}
x^\ell \mbox{ and } y^\ell \mbox{ are central in } k[x] \otimes_\tau k[y].
\end{equation} 
If the above is true, then the hypothesis of Theorem~\eqref{thm:centralize} holds for the subalgebras $A_0 = k[x^\ell]$ and $B_0 = k[y^\ell]$.
Therefore $\tau \colon k[y] \otimes^! k[x] \to k[x] \otimes^! k[y]$ is continuous, and the finite dual has the form
\begin{equation}\label{eq:twisted plane}
(k[x] \otimes_\tau k[y])^\circ \cong \Dist(\A^1_k) \otimes^\phi \Dist(\A^1_k)
\end{equation}
where the cotwisting map $\phi$ corresponds to $\tau^\circ \colon k[y]^\circ \otimes k[x]^\circ \to k[x]^\circ \otimes k[y]^\circ$ under the isomorphisms $k[x]^\circ \cong k[y]^\circ \cong \Dist(\A^1_k)$.

If we take $\tau = \sigma_{B,A}$ to be the ordinary tensor swap, then the cotwisting map $\phi$ corresponding to $(\sigma_{B,A})^\circ$ is also a tensor swap. In this case we have $A \otimes B \cong k[x,y]$ and of course
\[
(k[x] \otimes k[y])^\circ \cong \Dist(\A^1_k) \otimes \Dist(\A^1_k) \cong \Dist(\A^2_k).
\]
is the coalgebra of distributions on the affine plane. For other continuous choices of $\tau$, the finite dual of $k[x] \otimes_\tau k[y]$ can thus be viewed as the coalgebra of distributions on the ``noncommutative plane'' for which the twisted tensor product behaves as an algebra of functions.

\separate

Continue to assume~\eqref{eq:ellth powers} holds.
In order to find an explicit representation for $\tau$, it will be advantageous to use the  $\Z/\ell\Z$-grading on the polynomial algebra induced from its natural $\mathbb{N}$-grading via the monoid homomorphism $\mathbb{N} \hookrightarrow \Z \twoheadrightarrow \Z/\ell\Z$ as follows:
\[
k[t] = k[t^\ell] \oplus k[t^\ell]t \oplus \cdots \oplus k[t^\ell]t^{\ell-1}.
\]
The twisting map acts in a predictable way relative to the $\Z/\ell\Z$-gradings on $k[x]$ and $k[y]$. Elements $f \in k[x]$ and $g \in k[y]$ can be decomposed into homogeneous components of the form 
\[
f(x) = \sum_{i=0}^{\ell - 1} f_i x^i \qquad \mbox{and} \qquad
g(y) = \sum_{j=0}^{\ell - 1} g_j y^j,
\]
where each of the terms $f_i \in k[x^\ell]$ and $g_j \in k[y^\ell]$ are central. Then multiplication in $k[x] \otimes_\tau k[y]$ satisfies
\[
gf = \sum_{i,j = 0}^{\ell - 1} (g_j y^j) (f_i x^i) = \sum_{i,j = 0}^{\ell - 1} f_i (y^j x^i) g_j,
\]
so that the twisting map takes on the form
\begin{equation}\label{eq:finite order twist}
\tau(g \otimes f) = \sum_{i,j = 0}^\ell f_i \tau(y^j \otimes x^i) g_j.
\end{equation}
In this way, the twisting map $\tau$ is essentially described by the $\ell^2$ values $\tau(y^j \otimes x^i) \in k[x] \otimes k[y]$ for $i,j = 0, \dots, \ell-1$.

Some of the twisting maps below will be described in terms of $q$-numbers, $q$-factorials, and $q$-binomial coefficients. Because we are interested in evaluating these expressions at roots of unity, we recall their definitions more carefully than usual. Given a nonnegative integer $m$, the following define polynomials in $\Z[t]$ with nonnegative coefficients:
\begin{align*}
[m]_t &= 1 + t + \cdots + t^{m-1}, \\
[m]_t! &= [m]_t \, [m-1]_t \cdots [1]_t.
\end{align*}
Note that $[m]_t = (1-t^m)/(1-t)$.
Similarly, for any nonnegative integer $i \leq m$, it happens that the rational function
\begin{equation}\label{eq:t-binomial}
\genfrac{[}{]}{0pt}{0}{m}{i}_t = \frac{[m]_t!}{[i]_t! \, [m-i]_t!}
\end{equation}
is in fact an polynomial in $\Z[t]$ with nonnegative coefficients, as explained in~\cite[Section~1.7]{stanley}. Now if $q \in k^\times$ is any nonzero value in our field, evaluating the polynomials $[m]_t$ and $[m]_t!$ at $t = q$ respectively yield the \emph{$q$-number} and \emph{$q$-factorial} of $m$,
\begin{align*}
[m]_q &= 
1 + q + \cdots + q^{m-1}, \\
[m]_q ! &= 
 [m]_q \, [m-1]_q \cdots [1]_q, 
\end{align*}
which are elements of $k$. 
To obtain the \emph{$q$-binomial coefficient}, we \emph{first} simplify~\eqref{eq:t-binomial} to an integer polynomial and \emph{then} evaluate
\[
\genfrac{[}{]}{0pt}{0}{m}{i}_q = \left. \genfrac{[}{]}{0pt}{0}{m}{i}_t \, \right|_{t = q}
\]
at $t = q$ to produce an element of $k$. (This order of operations is particularly important if $q$ is a root of unity, which may be a root of the denominator of~\eqref{eq:t-binomial}. By first reducing, we obtain an unambiguous value for the $q$-binomial coefficient in all cases.)

\subsection{Quantum planes}\label{sub:quantum plane}

In this subsection we will analyze the dual coalgebra of the quantum plane at a root of unity. 
This coalgebra was also studied in~\cite[Subsection~4.2]{reyes:qmax} in terms of the representation theory of the quantum plane. Here we describe its comultiplication explicitly by viewing it as a cotwisted tensor product, which provides a satisfying view of the quantum plane as a deformation away from the classical plane. 
Below, we use the term ``deformation'' informally, not referring to a particular theory of algebraic deformation, but rather to mean the intuitive idea of structures that vary in with a change of parameter.

Let $q \in k^\times$. Recall that the (algebra of functions on the) \emph{quantum plane} is the affine domain
\[
\O_q(k^2) = k_q[x,y] = k \langle x, y \mid yx = q  xy \rangle.
\]
This is also the twisted tensor product $A \otimes_{\tau_q} B$ where $A = k[x]$ and $B = k[y]$, with twisting map
\begin{align*}
\tau_q \colon B \otimes A &\to A \otimes B \\
\tau_q(y^i \otimes x^j) &= q^{ij} x^j \otimes y^i.
\end{align*}
If $q$ is not a root of unity, then it is known~\cite[Example~II.1.2]{browngoodearl:quantum} that the only maximal ideals of finite codimenision in $\O_q(k^2)$ are of the form $(x, y-\lambda)$ or $(x-\lambda,y)$ for $\lambda \in k$. Using a strategy similar to that of Example~\ref{ex:not continuous}, one can verify in this case that $k_q[x,y] = A \otimes_{\tau} B$ is not homeomorphic to $A \otimes^! B$, so that $\tau_q$ is not continuous.

Suppose from now on that $q$ is a primitive $\ell$th root of unity, so that $\ell$ does not divide the characteristic of $k$. 
Letting $\theta$ denote the $k$-algebra automorphism of $k[x]$ given by $\theta(x) = qx$, then $\O_q(k^2) \cong k[x][y; \theta]$ satisfies the hypotheses of Corollary~\ref{cor:Ore}, and it follows (as in the proof of that result) that the twisting map $\tau$ above is continuous.
In this case the isomorphism~\eqref{eq:twisted plane} takes on the form 
\[
\O_q(k^2)^\circ \cong \Dist(\A^1_k) \otimes^{\phi_q} \Dist(\A^1_k)
\]
for a cotwisting map $\phi_q$ that corresponds to $\tau_q^\circ$.
Note that the underlying vector space is independent of $q$, while the cotwisting map $\phi_q$ varies with the root of unity~$q$. 

Thus as the algebras $\O_q(k^2)$ are deformed throughout the family, their spectral coalgebras have identical underlying vector spaces and counits, but their comultiplications vary with the choice of parameter~$q$. We interpret this by saying that the linear span of their quantum states remains unchanged, but that the quantum diagonal structure varies with~$q$ and causes the underlying quantum set to deform.  This picture fits quite intuitively within the framework of quantum groups and $q$-deformations~\cite[Chapter~I.1]{browngoodearl:quantum}.
Of course, when $q = 1$ then $\tau^\circ = (\sigma_{B,A})^\circ$ is simply a tensor swap, and the resulting coalgebra $\Dist(\A^1_k) \otimes \Dist(\A^1_k) \cong \Dist(\A^2_k)$ consists of distributions on the plane, recovering the affine plane as the ``limit'' of the quantum planes as $q \to 1$. 

\separate

Next we will describe the cotwisting map $\tau_q^\circ$ as a deformation away from the classical case $q = 1$ where $\tau_1 = \sigma_{B,A} \colon B \otimes A \to A \otimes B$. 
This is done in terms of $q$-numbers, as defined above.
Given $f \in A$ and $g \in B$ with $\Z/\ell\Z$-graded decompositions $f = \sum f_i \in k[x]$ and $g = \sum g_j \in k[y]$, we have
\begin{align*}
\tau(g \otimes f) &= \sum_{i,j = 0}^{\ell-1} \tau(g_j \otimes f_i) \\
&= \sum_{i,j = 0}^{\ell-1} q^{ij} f_i \otimes g_j \\
&= f \otimes g - \sum_{i,j=1}^{\ell-1} (1-q^{ij}) f_i \otimes g_j \\
&= f \otimes g - (1-q) \sum_{i,j=1}^{\ell-1} [ij]_q \, f_i \otimes g_j.
\end{align*}
Thus for $\sigma = \sigma_{B,A}$ we have 
\[
\tau_q = \sigma - (1-q) \xi_q,
\]
where we define the \emph{$q$-twist} $\xi_q \colon k[y] \otimes k[x] \to k[x] \otimes k[y]$ by the formula 
\[
\xi_q(g \otimes f) = \sum_{i,j=1}^{\ell-1} [ij]_q \, f_i \otimes g_j.
\]
(Note that $\xi_1 = 0$ since the sum is empty.) 
Because $\tau_q$ and $\sigma$ are both continuous, the same must be true for $\xi_q$. Then passing to the continuous dual, we have
\[
\tau_q^\circ = \sigma^\circ - (1-q) \xi_q^\circ.
\]
This explicitly represents the (co)twisting maps as deformations away from the classical case as $q$ varies.

\subsection{Quantized Weyl algebras}

Once again let $q \in k^\times$. The first \emph{quantized Weyl algebra} is defined as
\[
A_1^q(k) = k \langle x,y \mid yx - qxy = 1\rangle.
\]
For $q \neq 1$, this is an Ore extension and thus a twisted tensor product of the form 
\[
A^1_q(k) = k[x][y; \theta, \del_q] = k[x] \otimes_{\tau_q} k[y],
\]
 where $\theta$ is again the automorphism $\theta(f(x)) = f(qx)$ and $\del_q$ is the \emph{$q$-derivation} (or \emph{$q$-difference}) operator
\[
\del_q f(x) = \frac{f(qx) - f(x)}{qx - x}.
\]
The $q$-derivative is easily computed on monomials as 
\[
\del_q x^i = [i]_q \, x^{i-1},
\]
from which it is easy to verify the following commutation relation between $\sigma$ and $\del_q$ as linear operators:
\[
\del_q \theta = q \, \theta \del_q.
\]
Note that if $q$ is a primitive $\ell$th root of unity, then we have $\theta^\ell = \id_{k[x]}$ and $\del_q^\ell = 0$.

This $q$-commutation relation 
allows us evaluate the expression~\eqref{eq:Ore twist} for our particular $\theta$ and $\delta = \del_q$ in terms of $q$-binomial coefficients. 
Specifically, for $f \in k[x]$ we have
\begin{equation}\label{eq:q Weyl twist}
\tau_q(y^m \otimes f) = \sum_{i=0}^m \genfrac{[}{]}{0pt}{0}{m}{i}_q \theta^i \, \del_q^{m-i}(f) \otimes y^i,
\end{equation}
which can be verified by induction on $m$ with an application of the $q$-Pascal identity~\cite[(1.67)]{stanley}.

As in the case of the quantum plane, this twisting map is only continuous for those special values of~$q$ that result in a large center.

\begin{lemma}
Let $q \in k^\times$ with $q \neq 1$.
The twist $\tau_q$ of~\eqref{eq:q Weyl twist} is continuous as a map $k[y] \otimes^! k[x] \to k[x] \otimes^! k[y]$ if and only if $q$ is a root of unity. 
\end{lemma}

\begin{proof}
First suppose $q$ is not a root of unity. Then as described in~\cite[Example~4.1]{jordan:simple} the normal element $u = xy - 1/(1-q) = q^{-1}(yx - 1/(1-q))$ annihilates all simple right $A^1_q$-modules. It follows that $u$ acts nilpotently on all finite-dimensional $A^1_q$-modules, so that every open neighborhood of zero in the finite topology of $A^1_q$ contains a power of $u$. 
It is straightforward to verify that when every power $(xy)^n$ is written in normal form in terms of monomials $x^iy^j$, it has zero constant term. It follows that the constant term of each $u^n$, written in normal form, is nonzero. 

If $\tau_q$ were continuous, then by Proposition~\ref{prop:topological conditions} the natural map 
\[
k[x] \otimes^! k[y] \to k[x] \otimes_{\tau_q} k[y] = A^1_q
\]
given by the identity on underlying vector spaces
would be a homeomorphism, where all three algebras are equipped with their finite topologies. 
This would mean that the subspace $(x) \otimes (y)$ is open. But every element of this subspace has zero constant term when written in normal form, so it cannot contain any power of $u$. As this contradicts the above, it follows that $\tau_q$ is not continuous.

On the other hand, if $q$ is a primitive $\ell$th root of unity with $\ell > 1$, it is known that $x^\ell$ and $y^\ell$ are central in $A^1_q(k)$. (Indeed, $x^\ell$ is central because $\theta(x^\ell) = q^\ell x^\ell = x^\ell$ and $\del_q(x^\ell) = [\ell]_q \, x^{\ell-1} = 0$. With a bit more effort, evaluating~\eqref{eq:q Weyl twist} in the case $m = \ell$ one can see that all terms vanish except for $i = \ell$, demonstrating that $y^\ell$ is central.)
So~\eqref{eq:ellth powers} is satisfied, which implies that $\tau_q$ is continuous.
\end{proof}

Although the finite dual of the quantum plane could be expressed relatively simply in terms of the $q$-twist map, there does not seem to be such a straightforward expression in the case of quantized Weyl algebras. However, one can attempt to moderately simplify the formula~\eqref{eq:q Weyl twist} for $\tau_q$ using the method of~\eqref{eq:finite order twist}, as we now discuss.

Assume that $\tau_q$ is continuous, so that $q$ is an $\ell$th root of unity for $\ell \geq 2$. 
First notice that for integers $m,n \geq $, applying~\eqref{eq:q Weyl twist} for $f(x) = x^n$ gives
\begin{align*}
\tau_q(y^m \otimes x^n) &= \sum_{i=0}^m \genfrac{[}{]}{0pt}{0}{m}{i}_q \theta^i \, \del_q^{m-i}(x^n) \otimes y^i, \\
&= \sum_{i=0}^m \genfrac{[}{]}{0pt}{0}{m}{i}_q [n]_q! \, \cdots \, [n-m+i-1]_q !  \, \theta^i ( x^{n - m + i} ) \otimes y^i \\
&= \sum_{i=0}^m \genfrac{[}{]}{0pt}{0}{m}{i}_q \, \genfrac{[}{]}{0pt}{0}{n}{m-i}_q [i]_q! \, q^{i(n-m+i)} x^{n - m + i} \otimes y^i,
\end{align*}
For integers $m$, $n$, and $i \leq m$ as above, let us denote
\[
c_{m,n}^i =  \genfrac{[}{]}{0pt}{0}{m}{i}_q \, \genfrac{[}{]}{0pt}{0}{n}{m-i}_q [i]_q! \, q^{i(n-m+i)} \in k
\]

For arbitrary integers $m,n \geq 0$, use the division algorithm to write $m = m_1 \ell + m_0$ and $n = n_1 \ell + n_0$. Following the same reasoning as~\eqref{eq:finite order twist} and applying the formula above, we obtain
\begin{align*}
\tau(y^m \otimes x^n) &= x^{m_1 \ell} \, \tau(y^{m_0} \otimes x^{n_0}) \, y^{m_1 \ell} \\
&= \sum_{i=0}^{m_0} c_{m_0, n_0}^i \, x^{n_1 \ell + n_0 -  m_0 + i} \otimes y^{m_1 \ell + i} \\
&= \sum_{i=0}^{m_0} c_{m_0, n_0}^i \, x^{n -  m_0 + i} \otimes y^{m - m_0 + i}. 
\end{align*}
Although the sum above is still complicated, it has the advantage that the number of terms in the sum is bounded by~$\ell$.

\subsection{Jordan plane and Weyl algebra in positive characteristic}

We close with a brief discussion of two more examples in which the continuity of the twisting map depends on the characteristic of the base field rather than a ``quantizing'' parameter. 

Recall that the \emph{Jordan plane} is the algebra
\[
J = k \langle x, y \mid yx - xy = y^2 \rangle.
\]
This is again a twisted tensor product of the form $J = k[x] \otimes_\tau k[y]$, where $\tau$ is explicitly described in~\cite[Proposition~1.2]{shirikov:twogenerated} as follows:
\begin{equation}\label{eq:Jordan twist}
\tau(y^m \otimes x^n) = \sum_{i=0}^n \binom{n}{i} \frac{(m+n-i-1)!}{(m-1)!} x^i \otimes y^{m+n-i} 
\end{equation}
Note that the expression $\frac{(m+n-i-1)!}{(m-1)!}$ defines an integer. Viewed this way, the above formula is valid both in characteristic zero and positive characteristic.

\begin{lemma}
The twisting map $\tau$ of~\eqref{eq:Jordan twist} is continuous as a map $k[y] \otimes^! k[x] \to k[x] \otimes^! k[y]$ if and only if $k$ has  positive characteristic. 
\end{lemma}

\begin{proof}
If $\ch k = p > 0$, then it is shown in~\cite{shirikov:twogenerated} that $x^p$ and $y^p$ are central in $R$. In this case, $\tau$ is continuous.

So assume $\ch k = 0$. To prove that $\tau$ is not continuous, it suffices by Proposition~\ref{prop:topological conditions} to show that the naturally induced map 
\[
k[x] \otimes^! k[y] \to k[x] \otimes_\tau k[y] = J
\] 
is not a homeomorphism, where all three algebras are equipped with their cofinite topologies. It is shown in~\cite[Lemma~2.1]{iyudu} that $y$ acts nilpotently on every finite-dimensional representation of $J$. (While that reference requires the field to be algebraically closed, its proof is still effective for an arbitrary field $k$ after embedding it into an algebraic closure.) This means that every open ideal in $J$ contains a power of~$y$. Under the vector space identification $J = k[x] \otimes k[y]$, this means that every open subspace of $J$ contains $1 \otimes y^m$ for some $m \geq 0$.

On the other hand, for $\lambda \in k^\times$,the open subspace $(x) \otimes k[y] + k[x] \otimes (y - \lambda) \subseteq k[x] \otimes^! k[y]$ corresponds under the vector space identification $k[x] \otimes k[y] = k[x,y]$ to the ideal $(x,y-\lambda) \subseteq k[x,y]$. This ideal of the ordinary polynomial algebra contains no powers of $y$. So we have an open subspace of $k[x] \otimes^! k[y]$ that contains no element of the form $1 \otimes y^m$. So the continuous map $k[x] \otimes^! k[y] \to J$ cannot be a homeomorphism.
\end{proof}

In the case of a continuous twist, we may again apply the same reasoning as~\eqref{eq:finite order twist}. Suppose that $\ch k = p > 0$. For positive integers $m$, $n$, and $i < n$, define
\[
c_{m,n}^i = \binom{n}{i} \frac{(m+n-i-1)!}{(m-1)!} \in \mathbb{N}.
\]
Given arbitrary positive integers $m$ and $n$, use the division algorithm to write $m = m_1 p + m_0$ and $n = n_1p + n_0$. Then from~\eqref{eq:finite order twist} and~\eqref{eq:Jordan twist} we have
\begin{align*}
\tau(y^m \otimes x^n) &= x^{n_1 p} \, \tau(y^{m_0} \otimes x^{n_0}) \, y^{m_1 p} \\
&= \sum_{i=0}^{n_0} c_{m_0, n_0}^i \, x^{n_1 p + i} \otimes y^{m_1 p + m_0 + n_0 - i} \\
&= \sum_{i=0}^{n_0} c_{m_0, n_0}^i \, x^{n - n_0 + i} \otimes y^{m + n_0 - i},
\end{align*}
expressing the twist on pure tensors of monomials as a sum with at most~$p$ terms.

\separate

Finally, let $\tau$ from Example~\ref{ex:not continuous} be the twist $\tau$ producing the Weyl algebra $A^1(k) = k[x] \otimes_\tau k[y]$. Recall from that example that $\tau$ is not continuous if $k$ has characteristic zero. However, if $k$ has positive characteristic~$p$, then it is well known that $x^p$ and $y^p$ are central elements of $A^1(k)$. In this case, $\tau$ is continuous and 
\[
A^1(k)^\circ \cong \Dist(\A^1_k) \otimes^\phi \Dist(\A^1_k)
\]
where $\phi$ corresponds to $\tau^\circ$. 

If $\ch k = p > 0$, we can once more apply the method of~\eqref{eq:finite order twist}. Let $m$ and $n$ be nonnegative integers. This time if we set
\begin{align*}
c^i_{m,n} &= \binom{m}{i} \, n \, (n-1) \cdots (n-m+i-1) \\ 
&= \binom{m}{i} \binom{n}{m-i} \, i! \in \mathbb{N}
\end{align*}
for $i < m$,
then 
the twist can once more be expressed on monomials as follows. Use the division algorithm to write $m = m_1 p + m_0$ and $n = n_1p + n_0$. Then we have
\begin{align*}
\tau(y^m \otimes x^n) &= x^{n_1 p} \tau(y^{m_0} \otimes x^{n_0}) y^{m_1 p} \\
&= \sum_{i=0}^{m_0} c_{m_0, n_0}^i \, x^{n_1 p + n_0 - m_0 + i} \otimes y^{m_1 p + i} \\
&= \sum_{i=0}^{m_0} c_{m_0, n_0}^i \, x^{n - m_0 + i} \otimes y^{m -m_0 + i},
\end{align*}
where again the number of terms is at most~$p$.

\bibliography{twisted-dual-v2}

\end{document}